\title{Rainbow cycles through specified vertices}
\author{
Henry Liu
\\
\\
\normalsize
School of Mathematics\\
Sun Yat-sen University\\
Guangzhou 510275, China\\
\normalsize liaozhx5@mail.sysu.edu.cn\\
}
\date{4 May 2024}

\documentclass[11pt]{article}
\usepackage{amsmath,amssymb,amscd,amsthm}
\usepackage[dvips]{graphicx}
\usepackage{xcolor}
\newtheorem{theorem}                   {Theorem}%[section]
\newtheorem{thm}             [theorem] {Theorem}

\newtheorem{lem}             [theorem] {Lemma}

\newtheorem{prp}             [theorem] {Proposition}

\newtheorem{prob}            [theorem] {Problem}

\def\dcup{\,\dot\cup\,}

\begin{document}
\maketitle

\begin{abstract}
An edge-coloured cycle is \emph{rainbow} if the edges have distinct colours. Let $G$ be a graph such that any $k$ vertices lie in a cycle of $G$. The \emph{$k$-rainbow cycle index} of $G$, denoted by $crx_k(G)$, is the minimum number of colours required to colour the edges of $G$ such that, for every set $S$ of $k$ vertices in $G$, there exists a rainbow cycle in $G$ containing $S$. In this paper, we will first prove some results about the parameter $crx_k(G)$ for general graphs $G$. One of the results is a classification of all graphs $G$ such that $crx_k(G)=e(G)$, for $k=1,2$. We will also determine $crx_k(G)$ for some specific graphs $G$, including wheels, complete graphs, complete bipartite and multipartite graphs, and discrete cubes.\\
\\
\noindent\textbf{AMS Subject Classification (2020):} 05C15, 05C38\\
\\
{\bf Keywords:} Edge-colouring; rainbow coloured cycle; specified vertices
\end{abstract}

\section{Introduction}
In this paper, all graphs are simple and finite. Let $K_t$ and $C_t$ denote the complete graph and the cycle on $t$ vertices. Let $G$ be a graph with vertex set $V(G)$ and edge set $E(G)$. For disjoint sets $X,Y\subset V(G)$, let $E(X,Y)$ denote the edges of $G$ with one end-vertex in $X$ and the other in $Y$. Let $[X,Y]$ be the bipartite subgraph of $G$ with classes $X$ and $Y$, and edge set $E(X,Y)$. For $r\in\mathbb N$, an \emph{$r$-colouring} of $G$ is a function $\phi : E(G)\to \{1,2,\dots,r\}$, and $\phi$ is \emph{proper} if $\phi(e)\neq \phi(e')$ whenever $e$ and $e'$ are adjacent edges. We will often refer the values $1,2,\dots,r$ as \emph{colours}, and we may also call $\phi$ an \emph{edge-colouring} or \emph{colouring}. For $k\in \mathbb N$, a graph $G$ is \emph{$k$-connected} if $|V(G)|>k$, and for every $C\subset V(G)$ with $|C|\le k-1$, the subgraph $G-C$ is connected. For any other undefined terms in graph theory, the reader is referred to book of Bollob\'as \cite{Bol98}.

The topic of colourings of graphs has a long history which dates back to the famous \emph{Four-colour conjecture} from the 1850s. For edge-colourings, a landmark result is due to Vizing \cite{Viz64} in 1964 which states that $\chi'(G)\in\{\Delta(G),\Delta(G)+1\}$ for any graph $G$, where $\chi'(G)$ and $\Delta(G)$ are the \emph{edge-chromatic number} (or \emph{chromatic index}) and \emph{maximum degree} of $G$. A significant direction in the topic, introduced independently in the late 1970s by Vizing \cite{Viz76}; and by Erd\H{o}s, Rubin and Taylor \cite{ERT80}, is \emph{list colourings}. A famous open problem is the \emph{List colouring conjecture}, which states that $\chi'(G)=\chi'_\ell(G)$ for any graph $G$, where $\chi'_\ell(G)$ is the \emph{list edge-chromatic number} of $G$. This conjecture first appeared in print in a paper of Bollob\'as and Harris \cite{BH85} in 1985. There is a vast amount of literature about colourings of graphs, and a good starting reference is the book by Jensen and Toft \cite{JT95}.

Problems in graph theory concerning cycles through specified vertices have been well-studied. A well-known result of Dirac \cite{Dir60} from 1960, which can be proved by using Menger's theorem \cite{KM27}, states that for $k\ge 2$, any $k$ specified vertices in a $k$-connected graph lie in a cycle. Since then, many problems and results in this topic have appeared. Bondy and Lov\'asz \cite{BL81} proved that for $k\ge 2$, any $k-1$ specified vertices in a $k$-connected non-bipartite graph lie in an \emph{odd} cycle; and for $k\ge 3$, any $k$ specified vertices in a $k$-connected graph lie in an \emph{even} cycle. When considering the degrees of a graph, another famous result of Dirac \cite{Dir52} from 1952 states that, any graph of order $n\ge 3$ with minimum degree at least $\frac{n}{2}$ contains a Hamilton cycle. This theorem has been generalised in many directions. One generalisation is by Bollob\'as and Brightwell \cite{BB93}. They proved that, if $n\ge k\ge 3$ and $d\ge 1$ are such that $s=\big\lceil\frac{k}{\lceil n/d\rceil-1}\big\rceil\ge 3$, then for any $k$ vertices of degree at least $d$ in a graph of order $n$, there exists a cycle containing at least $s$ of the vertices.

An edge-coloured graph is \emph{rainbow} if all the edges have distinct colours. Another popular research topic considers the presence of rainbow coloured subgraphs in edge-coloured graphs. One direction is \emph{anti-Ramsey numbers}, which was initiated by Erd\H{o}s, Simonovits and S\'os \cite{ESS75} in 1973. See \cite{Alo83,JW03,MN05} for results that involve rainbow cycles. A more recent direction is \emph{rainbow Tur\'an numbers}, which was proposed by Keevash et al.~\cite{KMSV07} in 2007. Many problems about rainbow cycles in edge-coloured graphs have also been considered. See for example, \cite{AFR95,Ale06,BPV07,ENR83,FK08,HT86}.

In this paper, we consider the following problem. \emph{For $k\ge 1$, let $\mathcal F_k$ denote the family of all graphs $G$ with the property that any $k$ vertices of $G$ lie in a cycle. Consider an edge-colouring of $G\in\mathcal F_k$ such that, any $k$ vertices lie in a rainbow cycle. What is the minimum number of colours in such an edge-colouring?} Any edge-colouring satisfying the condition of the problem is called a \emph{$k$-rainbow cycle colouring}, and a $1$-rainbow cycle colouring is simply called a \emph{rainbow cycle colouring}. For $G\in\mathcal F_k$, the minimum number of colours in a $k$-rainbow cycle colouring of $G$ is denoted by $crx_k(G)$, and is called the \emph{$k$-rainbow cycle index} of $G$. Thus, $crx_k(G)$ is well-defined if and only if $G\in\mathcal F_k$. The parameter $crx_k(G)$ can be considered as a type of edge-chromatic number. Moreover, the problem of the study of $crx_k(G)$ can be classed under both of the ``cycles through specified vertices'' and ``rainbow cycles in graphs'' themes. 

There are conceivably some real-world applications. For example, suppose that a certain country is served by many different airlines, where some pairs of cities are connected directly by exactly one airline. A travelling salesman wants to attend business meetings in $k$ designated cities in no particular order, visiting each city once, and returning to his originating city. Between two consecutive designated cities, he can travel via other cities, and in the entire trip, he never visits any city more than once. What is the minimum number of airlines that the country needs so that, for any choice of $k$ designated cities, the travelling salesman can always complete his trip, using a different airline between any two consecutive cities that he visits? This minimum is precisely $crx_k(G)$, where $G$ is the graph whose vertex set consists of all the cities, and two vertices of $G$ are adjacent if and only if the corresponding two cities are connected directly by an airline.

%Indeed, for many problems in graph theory that involve trees, it is natural to formulate analogous problems that involve cycles, and vice versa. This is another motivation for us to introduce and study the function $crx_k(G)$, since it is the analogue of the function $rx_k(G)$ that involves cycles.

We note that Chartrand, Okamoto and Zhang \cite{COZ10} considered the \emph{$k$-rainbow index}, which is the analogous parameter with ``tree'' in place of ``cycle''. For a connected graph $G$ on $n$ vertices and $2\le k\le n$, the \emph{$k$-rainbow index}, denoted by $rx_k(G)$, is the minimum number of colours needed to colour the edges of $G$ so that, any $k$ vertices of $G$ are connected by a rainbow tree. Trivially, we also let $rx_1(G)=0$ for every graph $G$ (not necessarily connected). We remark that the $k$-rainbow cycle index $crx_k(G)$ has previously been studied by Li et al.~\cite{LSTZ19}, when they proved some computational complexity results. The problem of the study of the parameter $crx_k(G)$ was in fact originally introduced by the present author to the authors of \cite{LSTZ19}.

The family $\mathcal F_k$ contains some important subfamilies of graphs, including all Hamiltonian and pancyclic graphs on at least $k$ vertices. Also, for $k\ge 2$, Dirac's theorem \cite {Dir60} implies that $\mathcal F_k$ contains all $k$-connected graphs. Conversely, if $G\in\mathcal F_k$ ($k\ge 2$), then it is necessary that $G$ is $2$-connected. In Proposition \ref{crx12prp}, we characterise the families $\mathcal F_1$ and $\mathcal F_2$, where $\mathcal F_2$ is precisely the family of all $2$-connected graphs. For $k\ge 3$, no simple characterisation of the graphs of $\mathcal F_k$ is known, and only some sufficient conditions for graphs to be in $\mathcal F_k$ have been found \cite{BL81,Dir52}.

Some initial observations can be made. Clearly, we have $\mathcal F_1\supset\mathcal F_2\supset\mathcal F_3\supset\cdots$. If $G\in\mathcal F_k$, and $g(G)$ and $c(G)$ denote the \emph{girth} and \emph{circumference} of $G$, then we have $3\le g(G)\le crx_1(G)\le crx_2(G)\le\cdots\le crx_k(G)\le e(G)$ and $k\le \min(crx_k(G),c(G))$. %We will in fact see that $crx_1(G)\le c(G)$ (Theorem ***?***), but for $k\ge 2$, the functions $crx_k(G)$ and $c(G)$ are generally incomparable (Corollary ***?***); that is, we have $crx_k(G)\le c(G)$ for infinitely many graphs $G$, and likewise for $crx_k(G)\ge c(G)$. 
Also, we have $crx_k(G)>rx_k(G)$. We will see in Theorem \ref{septhm} that for any $k\ge 1$, the difference $crx_k(G)-rx_k(G)$ can be small (at most $3$) for infinitely many graphs $G$, and can also be arbitrarily large. If $G,H\in\mathcal F_k$, and $H$ is a spanning subgraph of $G$, then $crx_k(G)\le crx_k(H)$.

This paper is organised as follows. In Section \ref{gensec}, we shall study the parameter $crx_k(G)$ for general graphs $G$. In particular, we shall characterise the graphs $G\in\mathcal F_k$ such that $crx_k(G)=e(G)$ for $k=1,2$. In Section \ref{specsec}, we will compute exactly or find tight estimates for $crx_k(G)$ for some specific graphs $G$, including wheels, complete graphs, complete bipartite and multipartite graphs, and discrete binary cubes.

\section{General graphs}\label{gensec}

In this section, we shall prove some results about the parameter $crx_k(G)$ for general graphs $G$. We begin by recalling some well-known concepts about $2$-connected graphs. 

A \emph{block} of a graph is a subgraph which is either an isolated vertex, or a bridge (i.e., a cut-edge), or a maximal $2$-connected subgraph (i.e., not strictly contained in another $2$-connected subgraph). Any graph $G$ has a unique \emph{block decomposition} into blocks $B_1,\dots,B_p$ for some $p\ge 1$, such that the edge sets $E(B_1),\dots,E(B_p)$ form a disjoint union of $E(G)$, and any two blocks are either vertex-disjoint, or meet at a cut-vertex of $G$. An \emph{end-block} is a block which is incident with at most one cut-vertex. If $G$ is connected and is not a block, then $G$ has at least two end-blocks, each of which contains exactly one cut-vertex. For further details, see \cite{Bol78} (Ch.~I) and \cite{Bol98} (Ch.~III.2).

A graph $G$ is \emph{minimally $k$-connected} if $G$ is $k$-connected, and for every edge $e\in E(G)$, the graph $G-e$ is not $k$-connected. The following result describes the structure of minimally $2$-connected graphs.

\begin{thm}[\cite{Bol78}, Ch.~I, Theorem 3.1]\label{min2connthm}
Let $G$ be a minimally $2$-connected graph, and $xy\in E(G)$. Then, $G-xy$ has no isolated vertex and at least two blocks. The blocks of $G-xy$, say $B_1,\dots,B_p$ for some $p\ge 2$, can be ordered linearly so that $V(B_i)\cap V(B_{i+1})=\{v_i\}$ for $1\le i<p$, where $v_1,\dots,v_{p-1}$ are all the cut-vertices of $G-xy$. The end-blocks of $G-xy$ are $B_1$ and $B_p$. Moreover, we have $x\in V(B_1)\setminus\{v_1\}$ and $y\in V(B_p)\setminus\{v_{p-1}\}$. See Figure 1.
\end{thm}
\indent\\[-0.8cm]
\begin{figure}[htp]
\centering
\includegraphics[width=8cm]{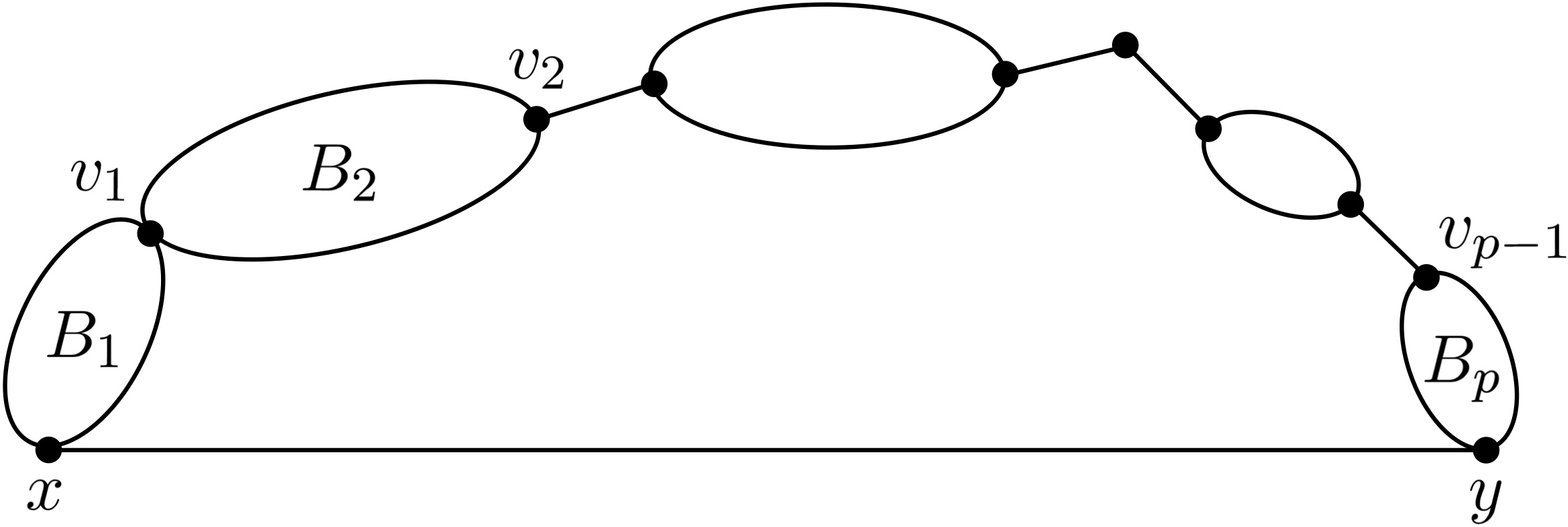}
\end{figure}
\indent\\[-0.8cm]
\begin{center}
Figure 1. The structure of $G$ in Theorem \ref{min2connthm}
\end{center}
\indent\\[-0.6cm]

%$x$\quad$y$\quad$B_1$\quad$B_2$\quad$B_p$\quad$v_1$\quad$v_2$\quad$v_{p-1}$\\
Finally, we have the following well-known result about the structure of $2$-connected graphs. See for example \cite{Die18}, Ch.~3, Proposition 3.1.1.
\begin{thm}\label{eardecompthm}
A graph is $2$-connected if and only if it can be obtained from a cycle, followed by successively attaching a non-trivial path by identifying the end-vertices with two distinct vertices of the previous graph.
\end{thm}

In Theorem \ref{eardecompthm}, each attached path is called an \emph{ear}, and the initial cycle, together with the ears, form an \emph{ear decomposition} for the $2$-connected graph. 

Now, we present the results. Recall that $crx_k(G)$ is well-defined if and only if $G\in\mathcal F_k$. We first give characterisations of the families $\mathcal F_1$ and $\mathcal F_2$. For $k\ge 3$, we have already mentioned that characterising the graphs of $\mathcal F_k$ is a much more difficult problem.

\begin{prp}\label{crx12prp}
\indent\\[-3ex]
\begin{enumerate}
\item[(a)] $\mathcal F_1$ is the family of all graphs with the property that, every vertex belongs to a $2$-connected block in the block decomposition.
\item[(b)] $\mathcal F_2$ is the family of all $2$-connected graphs. 
\end{enumerate}
\end{prp}

\begin{proof}
(a) Let $G$ be a graph with the property as stated, so that any vertex $v\in V(G)$ belongs to a $2$-connected subgraph. By Dirac's theorem \cite{Dir60}, $v$ lies in a cycle. Hence, $G\in\mathcal F_1$.

Conversely, suppose that $G$ does not have the stated property. If $G$ has an isolated vertex $u$, then $u$ does not lie in a cycle, and $G\not\in\mathcal F_1$. Otherwise, $G$ has a vertex $v$ such that, in the block decomposition of $G$, every block containing $v$ is a bridge of $G$. If $v$ lies in a cycle $C$ of $G$, then the two edges of $C$ incident with $v$ are bridges of $G$. This is a contradiction, since no bridge of $G$ can lie in a cycle of $G$. Hence, $v$ does not lie in a cycle of $G$, and $G\not\in\mathcal F_1$.\\[1ex]
\indent (b) By Menger's theorem \cite{KM27}, $G$ is $2$-connected if and only if any two vertices of $G$ have two internally vertex-disjoint paths connecting them. Equivalently, any two vertices of $G$ have a cycle containing them, which is equivalent to $G\in\mathcal F_2$.
\end{proof}

Next, recall that $crx_k(G)\le e(G)$ for $G\in\mathcal F_k$. We may ask: \emph{For what graphs $G$ do we have $crx_k(G)=e(G)$?} The analogous result for $rx_k(G)$ can be seen from Chartrand et al.~\cite{COZ10}: \emph{For a connected graph $G$ on $n$ vertices and $2\le k\le n$, we have $rx_k(G)=e(G)$ if and only if $G$ is a tree.} Here, we have the following result.
\newpage
\begin{thm}\label{crxk=m}
Let $G$ be a graph of order $n\ge 3$.
\begin{enumerate}
\item[(a)] Let $G\in\mathcal F_1$. Then $crx_1(G)=e(G)$ if and only if $G$ is the cycle $C_n$.
\item[(b)] Let $G\in\mathcal F_2$. Then $crx_2(G)=e(G)$ if and only if $G$ is minimally $2$-connected.
\item[(c)] Let $G\in \mathcal F_n$ (equivalently, $G$ is Hamiltonian), and $1\le k\le n$. Then $crx_k(G)=e(G)$ if and only if $G$ is the cycle $C_n$.
%Let $3\le k\le n$ and $G\in\mathcal F_k$ be Hamiltonian (in particular, this holds if $k=n$). Then $crx_k(G)=e(G)$ if and only if $G$ is the cycle $C_n$.
\end{enumerate}
\end{thm}

\begin{proof}
(a) If $G=C_n$, then clearly we have $crx_1(G)=e(G)$. 

Conversely, suppose that $G\neq C_n$. If $G$ is $2$-connected, then by Theorem \ref{eardecompthm}, $G$ has an ear decomposition with at least one ear. Let $P$ be the final ear attached, and $G'$ be the subgraph of $G$ before $P$ is attached. Let $u,v\in V(G')$ be the end-vertices of $P$, and $e\in E(P)$. Let $P'\subset G'$ be a $u-v$ path. Now, assign a rainbow colouring to $G-e$, and then assign a used colour not in $P\cup P'$ to $e$. We have a rainbow cycle colouring of $G$ with fewer than $e(G)$ colours, since every vertex of $G'$ lies in a rainbow cycle in $G'$ (since $G'$ is $2$-connected), and every vertex of $P$ lies in the rainbow cycle $P\cup P'$.

Otherwise, $G$ is not $2$-connected, so that the block decomposition of $G$ consists of at least two blocks. By Proposition \ref{crx12prp}(a), no block is an isolated vertex of $G$, and every vertex of $G$ belongs to a $2$-connected block. We colour the edges of $G$ in such a way that every block is rainbow coloured, and two of the blocks have a colour in common. This is a rainbow cycle colouring with fewer than $e(G)$ colours, since any vertex lies in a rainbow cycle in a $2$-connected block which contains the vertex.\\[1ex]
\indent (b) Suppose that $G$ is not minimally $2$-connected. Then there exists $e\in E(G)$ such that $G-e$ is $2$-connected. We assign a rainbow colouring to $G-e$, and then assign a used colour to $e$. This is a $2$-rainbow cycle colouring using fewer than $e(G)$ colours.

Conversely, suppose that $G$ is minimally $2$-connected, and we have a colouring of $G$ with fewer than $e(G)$ colours. Then there exist $e,e'\in E(G)$ with the same colour. By Theorem \ref{min2connthm}, the block decomposition of $G-e$ has the following structure. The blocks are $B_1,\dots, B_p$ for some $p\ge 2$, and can be ordered linearly so that $V(B_i)\cap V(B_{i+1})=\{x_i\}$ for $1\le i<p$, where $x_1,\dots,x_{p-1}$ are all the cut-vertices of $G-e$, and the end-blocks are $B_1$ and $B_p$. Moreover, we have $e=x_0x_p$, where $x_0\in V(B_1)\setminus\{x_1\}$ and $x_p\in V(B_p)\setminus\{x_{p-1}\}$.
%
%The graph $G-e$ is connected and has a cut-vertex. By taking the block decomposition of $G-e$, we cannot have more than two end-blocks. Otherwise, there exists an end-block $B$ of $G-e$ such that no vertex of $V(B)\setminus\{u\}$ is an end-vertex of $e$, where $u$ is the cut-vertex of $G-e$ that lies in $B$. But then, $u$ would be a cut-vertex of $G$, a contradiction. Hence, $G-e$ contains exactly two end-blocks, and its block decomposition must have the following structure: The blocks are $B_1,\dots, B_p$ for some $p\ge 2$, and can be ordered linearly so that $V(B_i)\cap V(B_{i+1})=\{x_i\}$ for $1\le i<p$, where $x_1,\dots,x_{p-1}$ are all the cut-vertices of $G-e$. The end-blocks are $B_1$ and $B_p$. Now, $e$ must have one end-vertex in $V(B_1)\setminus\{x_1\}$ and the other in $V(B_p)\setminus\{x_{p-1}\}$, otherwise, either $x_1$ or $x_{p-1}$ would be a cut-vertex of $G$. Assume that $e=x_0x_p$, where $x_0\in V(B_1)\setminus\{x_1\}$ and $x_p\in V(B_p)\setminus\{x_{p-1}\}$.
%
Let $e'\in E(B_\ell)$. We may assume that $1<\ell\le p$. If $B_\ell$ is a bridge of $G-e$, then $e'=x_{\ell-1}x_{\ell}$, and any cycle of $G$ containing $x_0$ and $x_{\ell}$ must use both $e$ and $e'$. Otherwise, we have $B_\ell$ is minimally $2$-connected, since $G$ is minimally $2$-connected. Again by Theorem \ref{min2connthm}, the blocks of $B_\ell-e'$ can be linearly ordered as $D_1,\dots, D_q$ for some $q\ge 2$, where $V(D_j)\cap V(D_{j+1})=\{y_j\}$ for $1\le j<q$, and $y_1,\dots,y_{q-1}$ are all the cut-vertices of $B_\ell-e'$. Moreover, we have $e'=y_0y_q$, where $y_0\in V(D_1)\setminus\{y_1\}$ and $y_q\in V(D_q)\setminus\{y_{q-1}\}$. Now, we have $x_{\ell-1},x_\ell\in V(B_\ell)$. We may assume that $x_{\ell-1}\in V(D_s)$ and $x_\ell\in V(D_t)$ for some $1\le s\le t\le q$. Note that $s$ and $t$ may possibly be not unique. If we must have $s=1$ and $t=q$, then $G-e'$ would be $2$-connected, a contradiction. If $1\le s\le t<q$, for a cycle $C$ to contain $x_0$ and $y_q$, emanating from $x_0$, we must use the edge $e$ and a path in $B_1$ connecting $x_0$ and $x_1$; and emanating from $y_q$, we must use the edge $e'$ and a path in $D_q$ connecting $y_q$ and $y_{q-1}$. Hence $C$ must use both $e$ and $e'$. If $1< s\le t\le q$, for a cycle $C$ to contain $x_0$ and $y_0$, again $C$ must use $e$; and emanating from $y_0$, we must use the edge $e'$ and a path in $D_1$ connecting $y_0$ and $y_1$. Again, $C$ must use both $e$ and $e'$.

We conclude that there are always two vertices of $G$ such that, any cycle in $G$ containing them must use both $e$ and $e'$. Hence, there is no $2$-rainbow cycle colouring with fewer than $e(G)$ colours, and  $crx_2(G)=e(G)$.\\[1ex]
\indent (c)  If $G=C_n$, then clearly we have $crx_k(G)=e(G)$. Conversely, suppose that $G\neq C_n$. Since $G$ is Hamiltonian, we have $e(G)>n$. By taking an $n$-colouring of $G$ such that a Hamilton cycle is rainbow, we have $crx_k(G)\le n$. Therefore, $crx_k(G)<e(G)$.
\end{proof}

For $3\le k<n$, the problem of classifying all graphs $G\in\mathcal F_k$ of order $n$ with $crx_k(G)=e(G)$ seems to be more difficult. The case when $k=n-1$ seems interesting. In this case, if $G$ is Hamiltonian, then we must have $G=C_n$. Otherwise, $G$ must be hypohamiltonian, i.e., $G$ is not Hamiltonian, and $G-v$ is Hamiltonian for every $v\in V(G)$. Classifying hypohamiltonian graphs with a fixed order is a problem of significant interest. The hypohamiltonian graph with the smallest order is the Petersen graph ${\sf P}_{10}$, with $10$ vertices. Aldred, McKay and Wormald \cite{AMW97} found all hypohamiltonian graphs of order at most $17$. There is one such graph when the order is $10$ (the Petersen graph), $13$ or $15$, there are four graphs of order $16$, and none with orders $11,12,14,17$. These graphs are shown in Figure 2.
\indent\\[-0.1cm]
\begin{figure}[htp]
\centering
\includegraphics[width=15cm]{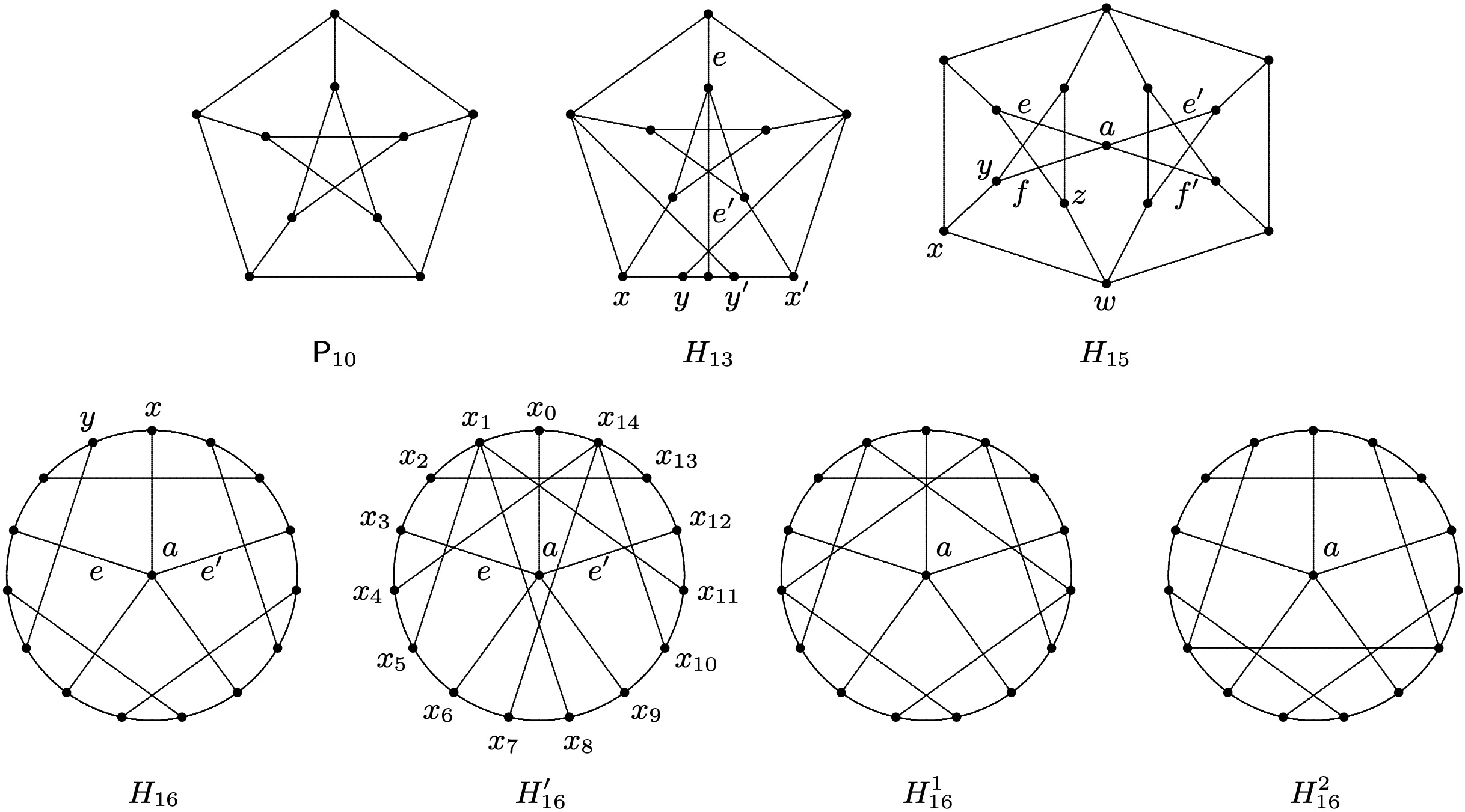}
\end{figure}
\indent\\[-0.8cm]
\begin{center}
Figure 2. Hypohamiltonian graphs with at most $17$ vertices
\end{center}
\indent\\[-0.2cm]
For the four graphs of order $16$, an edge connecting the centre vertex $a$ to the outside cycle is a \emph{spoke edge}. In the next result, we show that, with the exception of the Petersen graph, we have $crx_{n-1}(G)<e(G)$ if $G$ is one of the graphs in Figure 2 and has order $n$.

\begin{thm}\label{hypothm}
\indent\\[-3ex]
\begin{enumerate}
\item[(a)] If ${\sf P}_{10}$ is the Petersen graph, then $crx_9({\sf P}_{10})=e({\sf P}_{10})=15$.
\item[(b)] If $G$ is a hypohamiltonian graph of order $n$, where $11\le n\le 17$, then $crx_{n-1}(G)<e(G)$.
\end{enumerate}
\end{thm}
%
%it may be conceivable that if $G\in\mathcal F_k$ is a graph of order $n$, then $crx_k(G)=e(G)$ if and only if $G=C_n$. However, this turns out to be false, as the next lemma shows.
%
%\begin{lemma}\label{P10lemma}
%Let $G$ be the Petersen graph (see Figure 2). Then $crx_9(G)=e(G)=15$.
%\end{lemma}
%
\begin{proof}
(a) Suppose that we have a colouring of ${\sf P}_{10}$, using fewer than $e({\sf P}_{10})=15$ colours. Then, two edges $e, e'$ have the same colour. We choose a vertex $v$ as follows. If $e$ and $e'$ are adjacent, say at $u$, then we let $v$ be the vertex adjacent to $u$ and not in $e$ or $e'$. If $e$ and $e'$ are not adjacent, then we choose $v$, not in $e$ or $e'$, which is adjacent to an end-vertex of $e$ and $e'$. Now, since ${\sf P}_{10}$ is hypohamiltonian, any cycle of ${\sf P}_{10}$ containing all vertices of ${\sf P}_{10}-v$ must use $e$ and $e'$, and hence is not rainbow. Therefore, $crx_9({\sf P}_{10})=e({\sf P}_{10})=15$.\\[1ex]
\indent(b) We show that if $G$ is one of the remaining graphs in Figure 2, with order $n$, then $crx_{n-1}(G)<e(G)$. Since $H_{16}$ is a spanning subgraph of both $H_{16}^1$ and $H_{16}^2$, the results for $H_{16}^1$ and $H_{16}^2$ follow from that of $H_{16}$. For the remaining four graphs, consider the colouring of $G$ where the edges $e,e'$ as shown in Figure 2 have the same colour, and the remaining edges have further distinct colours. We claim that this is an $(n-1)$-rainbow cycle colouring for $G$. It suffices to show that $G-v$ contains a Hamilton cycle not using $e$ or $e'$ for every $v\in V(G)$, which therefore is a rainbow cycle. Clearly, it suffices to check this property when $v$ is not an end-vertex of $e$ or $e'$.

First, consider $G=H_{13}$. Note that $H_{13}$ contains a subdivided Petersen graph as a spanning subgraph, and not using $e'$. Then, $H_{13}-v$ contains a Hamilton cycle not using $e'$ for $v\in V(H_{13})\setminus\{x,x',y,y',z\}$. Figure 3(a) shows that both $H_{13}-x$ and $H_{13}-y$ (and thus also $H_{13}-x'$ and $H_{13}-y'$ by symmetry) contains a Hamilton cycle, not using $e$ or $e'$.

Next, consider $G=H_{15}$. Figure 3(b) shows that each of $H_{15}-w, H_{15}-x, H_{15}-y$ and $H_{15}-z$ contains a Hamilton cycle, using either $e,f'$ or $e',f$. Hence by symmetry, $H_{15}-v$ contains a Hamilton cycle not using $e$ or $e'$, for every $v\in V(H_{15})\setminus\{a\}$.

Next, consider $G=H_{16}$. Figure 3(c) shows that both $H_{16}-x$ and $H_{16}-y$ contains a Hamilton cycle that uses exactly two consecutive spoke edges. By symmetry, the same is true for $H_{16}-v$, for every $v\in V(H_{16})\setminus\{a\}$. Any such cycle does not use $e$ or $e'$. 

Finally, consider $G=H_{16}'$. Figure 3(d) shows that $H_{16}'-x_i$ contains a Hamilton cycle not using $e$ or $e'$, for $0\le i\le 3$. We have the automorphisms $\tau$ and $\sigma$ on $H_{16}'$, where $\tau$ is the obvious symmetry; and $\sigma$ fixes $a,x_0,x_1,x_{14}$, $\sigma(x_i)=x_{i+3}$ for $2\le i\le 10$, and $\sigma(x_i)=x_{i-9}$ for $11\le i\le 13$. By using a composition of $\tau$ and $\sigma$, for $4\le i\le 14$, it is easy to obtain a Hamilton cycle of $H_{16}'-x_i$ which uses the spoke edge $ax_0$, and hence not using $e$ or $e'$.
%
%, where the last three cycles all use the spoke edge $az$. Next, observe that $H_{16}'$ has the automorphisms $\sigma$ which fixes $a$ and $z$, and switches $x_i$ and $y_i$ for $1\le i\le 7$; and $\tau$ which fixes $a,z,x_1$ and $y_1$, with the remaining $12$ vertices permuted as follows:
%\[
%(\,x_2 \:\:\, y_4 \:\:\, y_7 \:\:\, x_5\,)\:(\,x_3 \:\:\, y_3 \:\:\, y_6 \:\:\, x_6\,)\:(\,x_4 \:\:\, y_2 \:\:\, y_5 \:\:\, x_7\,).
%\]
%Note that each of $\sigma$ and $\tau$ fixes the spoke edge $az$. Now for every $v\in V(H_{16}')\setminus\{a,z\}$, there exists $1\le i\le 3$ and an automorphism $\rho=\sigma^p\circ\tau^q$, for some $0\le p\le 1$ and $0\le q\le 3$, such that $\rho(x_i)=v$. The Hamilton cycle of $H_{16}'-x_i$ is then mapped by $\rho$ to a Hamilton cycle of $H_{16}'-v$ which also uses $az$. Thus, the Hamilton cycle of $H_{16}'-v$ can use at most one of $e,e'$.
\end{proof}
\indent\\[-0.8cm]
\begin{figure}[htp]
\centering
\includegraphics[width=14.9cm]{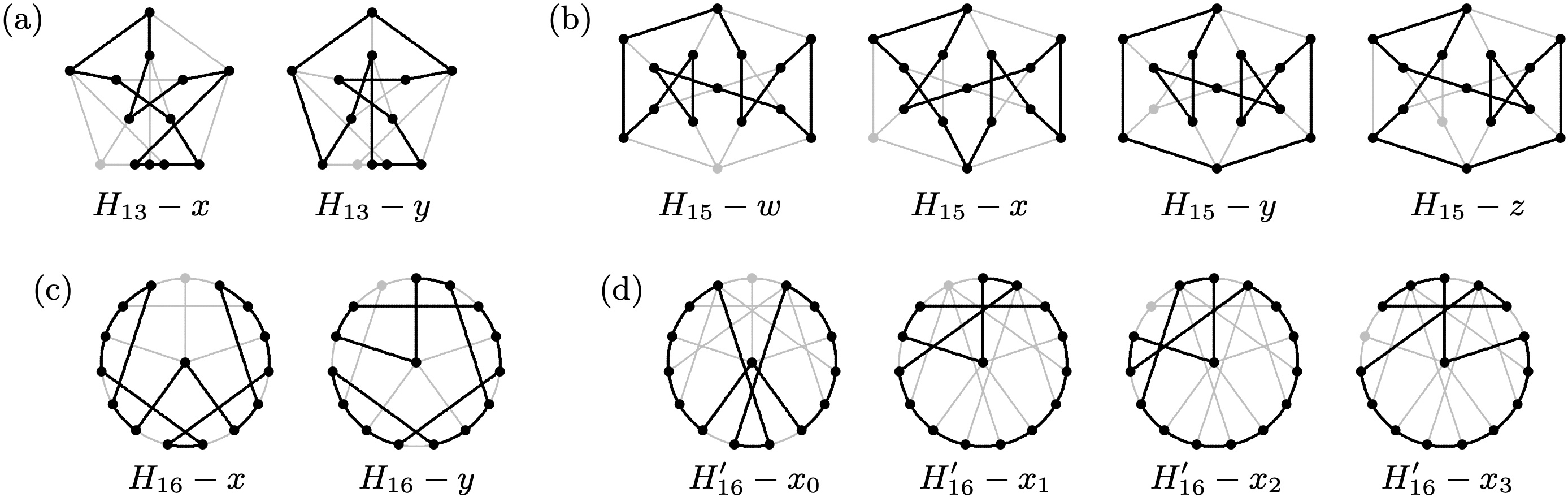}
\end{figure}
\indent\\[-0.8cm]
\begin{center}
Figure 3. Cycles in hypohamiltonian graphs
\end{center}
\indent\\[-0.2cm]
%
%Lemma \ref{P10lemma} provides an example which shows that in Theorem \ref{crxk=m}(c), the same conclusion does not hold if we strengthen the assumption to $G\in\mathcal F_{n-1}$ and $1\le k\le n-1$. We pose the following problem.
%
We propose the following problem.
\begin{prob}
Let $3\le k<n$. Characterise the graphs $G\in\mathcal F_k$ of order $n$ with $crx_k(G)=e(G)$. In particular, when $k=n-1$, does there exist $G$ such that $crx_{n-1}(G)=e(G)$, other than when $G$ is the cycle $C_n$, or $n=10$ and $G$ is the Petersen graph?
\end{prob}
%
%\begin{thm}
%$G\in \mathcal F_1$. Then $crx_1(G)\le c(G)$. Equality iff every $2$-connected block of $G$ is a cycle ***Not true: $K_{2,n}$?***.
%\end{thm}
%
%Improves Thm 3(a). If $G\in \mathcal F_1$, then $c(G)\le e(G)$, since $c(G)\le n$ and $\delta(G)\ge 2$, so $e(G)\ge n$. $c(G)=e(G)$ iff $G$ is a cycle.
%
%\begin{proof}
%Since $G\in\mathcal F_1$, it follows that there exist cycles $H_1,\dots,H_p\subset G$, for some $p\ge 1$, such that $V(H)=V(G)$, where $H=H_1\cup\cdots\cup H_p$. We may assume that $c(G)\ge e(H_1)\ge\cdots\ge e(H_p)$. We are done if we can construct an edge-colouring of $H$ with at most $e(H_1)$ colours such that, every $H_i$ is rainbow coloured. We proceed inductively. First, we give $H_1$ a rainbow colouring. Now for $1<j\le p$, suppose that we have constructed an edge-colouring of $H_1\cup\cdots\cup H_{j-1}$ with at most $e(H_1)$ colours such that, $H_1,\dots,H_{j-1}$ are all rainbow coloured. Now, we partition $E(H_j)=E_1\dcup E_2\dcup E_3$, where $E_1=E(H_1\cap H_j)$, $E_2=(E(H_2\cup\cdots\cup H_{j-1})\setminus E(H_1))\cap E(H_j)$ and $E_3=(E(H_j)\setminus E(H_1\cup\cdots\cup H_{j-1})$. Note that the edges of $E_1\cup E_2$ are already coloured, while those of $E_3$ are not yet coloured.
%\end{proof}
%
%\begin{cor}
%$k\ge 2$, $crx_k(G)$ and $c(G)$ incomparable.
%\end{cor}
%$crx_k(G)\ll c(G)$: take $K_n$. $crx_k(G)\gg c(G)$: take $K_{k,n}$.

We conclude this section by comparing the parameters $crx_k(G)$ and $rx_k(G)$. It is easy to see that for $k\ge 1$ and $G\in\mathcal F_k$, we have $crx_k(G)>rx_k(G)$. Indeed, we may take a $k$-rainbow cycle colouring of $G$ with $t=crx_k(G)\ge 3$ colours, and recolour all edges of colour $t$ with colour $1$. In the new $(t-1)$-colouring, any $k$ vertices lie on a rainbow path. Hence, $rx_k(G)\le crx_k(G)-1$.

How close and how far apart can the parameters $crx_k(G)$ and $rx_k(G)$ be? In Theorem \ref{septhm} below, we see that the difference $crx_k(G)-rx_k(G)$ can be small (at most $3$) for infinitely many graphs, and can also be arbitrarily large.

\begin{thm}\label{septhm}
Let $k\ge 1$.
\begin{enumerate}
\item[(a)] There are infinitely many graphs $G$ such that 
\[
crx_k(G)-rx_k(G)=
\left\{
\begin{array}{l@{\quad}l}
3, & \textup{\emph{if $k=1$,}}\\[0.5ex]
2, & \textup{\emph{if $k=2,3$,}}\\[0.5ex]
1, & \textup{\emph{if $k\ge 4$.}}
\end{array}
\right.
\]
\item[(b)] For every $c\ge 1$, there exists a graph $G$ such that $rx_k(G)\le k^2-1$ and $crx_k(G)\ge c$. Hence, the difference $crx_k(G)-rx_k(G)$ can be arbitrarily large.
\end{enumerate}
\end{thm}

\begin{proof}
(a) In Theorem \ref{crxkKnthm} in Section \ref{specsec}, we will see that $crx_1(K_n)=crx_2(K_n)=3$ for $n\ge 3$. Hence, $crx_1(K_n)-rx_1(K_n)=3$ and $crx_2(K_n)-rx_2(K_n)=2$ for $n\ge 3$.

A result of Chartrand et al.~(\cite{COZ10}, Theorem 2.1) states that $rx_3(C_n)=n-2$ for $n\ge 4$, and $rx_k(C_n)=n-1$ for $4\le k\le n$. Hence, $crx_3(C_n)-rx_3(C_n)=2$ for $n\ge 4$, and $crx_k(C_n)-rx_k(C_n)=1$ for $4\le k\le n$.\\[1ex]
\indent (b) Let $P$ be a path on $k-1$ vertices, and $C$ be a cycle on $kt\ge 3$ vertices. Let $G=P+C$, the join of $P$ and $C$. Note that for $k=1$, $G$ is the cycle $C_t$. Let $V(P)=\{u_0,\dots, u_{k-2}\}$, and $V(C)=\{v_0,v_1,\dots,v_{kt-1}\}$, with the indices of the vertices of $C$ taken cyclically modulo $kt$. If $k=1$, then $crx_1(G)=t\ge c$ if $t$ is sufficiently large, and $rx_1(G)=0$. From now on, let $k\ge 2$.

First, let $S=\{v_0,v_t,v_{2t}\dots,v_{(k-1)t}\}$. It is easy to see that any cycle of $G$ containing $S$ must contain the vertices $v_{it},v_{it+1},\dots, v_{(i+1)t-1}$, for some $0\le i\le k-1$. Hence, we have $crx_k(G)\ge t\ge c$ if $t$ is sufficiently large.

Now, define a colouring $\phi$ of $G$, using colours $0,1,2,\dots, k^2-2$, as follows. For $0\le i\le kt-1$, let $0\le \bar{i}\le k-1$ be the integer with $\bar{i}\equiv i$ (mod $k)$. For $0\le j\le k-2$, let $\phi(u_jv_i)=\bar{i}+jk$. For $0\le j\le k-3$, let  $\phi(u_ju_{j+1})=k^2-k+j$. Let $\phi(e)=k^2-2$ for every edge $e$ of $C$. We claim that any $k$ vertices are connected by a rainbow tree. It suffices to consider a set $S$ of $k$ vertices in $C$, and find a rainbow tree $T$ containing $S$ and using all edges of $P$. Let $S_1,\dots,S_r$ be the sets $S\cap\{v_{\ell k},v_{\ell k+1},\dots,v_{(\ell+1)k-1}\}$ which are non-empty, for $0\le\ell\le t-1$. Note that $1\le r\le k$. If $1\le r\le k-1$, then we obtain $T$ by taking $P$ and adding the edges $u_jv_i$, where $0\le j\le r-1$ and $v_i\in S_{j+1}$. If $r=k$, then we have $|S_1|=\cdots=|S_k|=1$. We obtain $T$ by taking $P$, and adding the edge from $u_j$ to $S_{j+1}$ for $0\le j\le k-2$. Let $S_{k-1}=\{v_a\}$ and $S_k=\{v_b\}$. If $a\not\equiv b$ (mod $k)$, then we also add the edge $u_{k-2}v_b$. If $a\equiv b$ (mod $k)$, then we also add the path $v_bv_{b-1}u_{k-2}$ if $b\not\equiv 0$ (mod $k)$, or the path  $v_bv_{b+1}u_{k-2}$ if $b\equiv 0$ (mod $k)$. In every case, $T$ is a suitable rainbow tree. Hence, $rx_k(G)\le k^2-1$. 
\end{proof}
%
%
%\section{$k$-connected graphs}
%Let $k\ge 2$. For a $k$-connected graph $G$ and any set of $k$ vertices of $G$, recall that Dirac's Theorem \cite{Dir60} states that $G$ contains a cycle passing through the $k$ vertices. Hence, $crx_k(G)$ is well-defined. Also, note that $crx_1(G)$ is well-defined for any $2$-connected graph $G$. Therefore, we may ask the following question.
%
%\begin{prob}
%Let $\ell\ge 2$ and $1\le k\le \ell$. What is the minimum integer $f_{k,\ell}(n)$ such that, $crx_k(G)\le f_{k,\ell}(n)$ for every $\ell$-connected graph $G$ on $n>\ell$ vertices?
%\end{prob}
%
%
%A result of Mader \cite{WM72} implies that any minimally $k$-connected graph on $n$ vertices has at most $kn$ edges. If $G$ is an $\ell$-connected graph on $n$ vertices, where $\ell\ge k$, then by considering a minimally $k$-connected spanning subgraph of $G$, we have $crx_k(G)\le kn$. 
%
%We have $c(2,2)=2$, since $crx_2(K_{2,n-2})=2n-4$ by Proposition \ref{crx2Knnthm}. 
%
%***Next case: $c(2,3)=\:?$***
%
%***Can similarly consider $c(1,\ell)$ for $\ell\ge 2$. Get $crx_1(G)\le 2n$ for any $\ell$-connected graph $G$ on $n$ vertices. First cases: $c(1,2)=\:?$, $c(1,3)=\:?$***
%
%***?: If $crx_1(G)\le c(G)$ for all $G\in\mathcal F_1$, then this gives $c(1,2)=1$ (UB: $c(G)\le n$; LB: Take $C_n$).
%
%$f_{k,\ell}(n)\ge f_{k,\ell+1}(n)$
%

\section{Specific families of graphs}\label{specsec}

We will now compute $crx_k(G)$ for some specific graphs $G$. First, we consider the case when $G$ is a wheel $W_n=C_n+v$, i.e., the graph on $n+1\ge 4$ vertices, obtained by connecting the \emph{centre} vertex $v$ to all vertices of the cycle $C_n$. Even though the connectivity of $W_n$ is only $3$, we see that $crx_k(W_n)$ is well-defined for every $k$ with $1\le k\le n+1$, since $W_n$ is Hamiltonian.

\begin{thm}
\indent\\[-3ex]
\begin{enumerate}
\item[(a)] $crx_1(W_n)=3$ for $n\ge 3$.
\item[(b)] $crx_2(W_3)=3$, and $crx_2(W_n)=\lceil\frac{n}{2}\rceil+2$ for $n\ge 4$.
\item[(c)] 
\[
crx_3(W_n)=
\left\{
\begin{array}{l@{\quad}l}
n, & \textup{\emph{if $3\le n\le 7$,}}\\[0.5ex]
n-1, & \textup{\emph{if $8\le n\le 11$,}}\\[0.5ex]
n-2, & \textup{\emph{if $n\ge 12$.}}
\end{array}
\right.
\]
\item[(d)] For $k\ge 4$ and $n+1\ge k$,
\[
crx_k(W_n)=
\left\{
\begin{array}{l@{\quad}l}
n+1, & \textup{\emph{if $n<2k$,}}\\[0.5ex]
n, & \textup{\emph{if $n\ge 2k$.}}
\end{array}
\right.
\]
\end{enumerate}
\end{thm}

\begin{proof}
Throughout, let $W_n=C_n+v$, where $v$ is the centre of $W_n$. Let $V'=\{v_0, v_1,\dots,v_{n-1}\}$ be the vertex set of $C_n$, where the vertices are in that order, with the indices taken modulo $n$. For a set $I$ of consecutive vertices of $V'$, where $2\le |I|\le n-1$, let $C(I)$ be the cycle of $W_n$ with vertex set $I\cup \{v\}$. %We also call such a cycle $C(I)$ is a \emph{$v$-cycle}, since it contains $v$. 
Let $\bar{I}=V'\setminus I$ and $\bar{I}_i=V'\setminus\{v_i\}$.\\[1ex]
\indent (a) It suffices to prove that $crx_1(W_n)\le 3$. Colour the edge $vv_i$ with colour $1$ if $i$ is odd, and colour $2$ if $i$ is even. Colour all edges of $C_n$ with colour $3$. This is a rainbow cycle colouring for $W_n$, since every vertex lies in a rainbow triangle.\\[1ex]
\indent (b) The colouring of $W_3$ where $vv_i$ and $v_{i+1}v_{i+2}$ have colour $i$, for $i=1,2,3$, is a $2$-rainbow cycle colouring. Hence, $crx_2(W_3)=3$.

Let $n\ge 4$. We define a colouring $\phi$ on $W_n$ as follows. The cycle $C_n$ uses colours $1,2,\dots,\lceil\frac{n}{2}\rceil$, where $\phi(v_{i-1}v_i)\equiv i$ (mod $\lceil\frac{n}{2}\rceil$). Let $\phi(vv_i)=\lceil\frac{n}{2}\rceil+1$ for $0\le i\le \lceil\frac{n}{2}\rceil-1$, and $\phi(vv_i)=\lceil\frac{n}{2}\rceil+2$ for $\lceil\frac{n}{2}\rceil\le i\le n-1$. Then, the cycles $C(I)$ where $|I|=\lfloor\frac{n}{2}\rfloor+1$, except when $n$ is odd and $I=\{v_0,\dots,v_{\lfloor n/2\rfloor}\}$; and the cycle $C(\{v_0,\dots,v_{\lceil n/2 \rceil}\})$, are all rainbow coloured. Any two vertices of $W_n$ lie in one of these cycles. Hence, $\phi$ is a $2$-rainbow cycle colouring, and $crx_2(W_n)\le\lceil\frac{n}{2}\rceil+2$.

%Indeed, for every $i$, the vertices $v$ and $v_i$ lie on a rainbow cycle consisting of $vv_0$, $vv_{\lceil n/2\rceil}$, and a path of $C_n$ connecting $v_0$ and $v_{\lceil n/2\rceil}$. For two vertices $v_i$ and $v_j$ with $i<j$, the same is true if $i=0$, or $1\le i<j\le \lceil\frac{n}{2}\rceil$, or $\lceil\frac{n}{2}\rceil\le i<j\le n-1$. Otherwise, $v_i$ and $v_j$ lie on the rainbow cycle consisting of $vv_i$, $vv_j$, and the shorter of the two paths of $C_n$ connecting $v_i$ and $v_j$. Hence, $crx_2(W_n)\le\lceil\frac{n}{2}\rceil+2$.

Now, we prove that $crx_2(W_n)\ge\lceil\frac{n}{2}\rceil+2$. The shortest cycle in $W_n$ containing $v_0$ and $v_{\lfloor n/2\rfloor}$ has length $\lfloor\frac{n}{2}\rfloor+2$. Hence $crx_2(W_n)\ge\lfloor\frac{n}{2}\rfloor+2$, and we are done for $n$ even. Let $n\ge 5$ be odd. Suppose that there exists a $2$-rainbow cycle colouring of $W_n$,  using exactly $\lfloor\frac{n}{2}\rfloor+2=\lceil\frac{n}{2}\rceil+1$ colours. For $0\le i\le n-1$, the only shortest cycle in $W_n$ containing $v_i$ and $v_{i+\lfloor n/2\rfloor}$ is $C(J_i)$, where $J_i=\{v_i,\dots,v_{i+\lfloor n/2\rfloor}\}$. Hence every cycle $C(J_i)$ must be rainbow coloured, using every one of the $\lfloor\frac{n}{2}\rfloor+2$ colours. By considering $C(J_i)$ and $C(J_{i-\lfloor n/2\rfloor})$, the only edge of $C_n$ that can possibly have the same colour as $vv_i$ is $v_{i+\lfloor n/2\rfloor}v_{i+\lfloor n/2\rfloor+1}$. Since $\lfloor\frac{n}{2}\rfloor+2<n$ for $n\ge 5$, assume that there exist two edges $vv_i$ and $vv_j$ with colour $1$. Then, no edge of $C_n$ has colour $1$. Since $vv_i,vv_{i+\lfloor n/2\rfloor}\in E(C(J_i))$ for all $0\le i\le n-1$, and $n$ is odd, there exists $0\le\ell\le n-1$ such that $vv_\ell$ and $vv_{\ell+\lfloor n/2\rfloor}$ have different colours, and different from colour $1$. But then, $C(I_\ell)$ contains an edge in $C_n$ using colour $1$, a contradiction. Hence, $crx_2(W_n)\ge\lceil\frac{n}{2}\rceil+2$.\\[1ex]
\indent (c) We have $crx_3(W_3)=3$ by considering the same colouring of $W_3$ as in (b). Let $n\ge 4$. We define a colouring $\phi$ on $W_n$ as follows. 
\begin{itemize}
\item For $4\le n\le 7$, colour the edges of $C_n$ with colours $1,\dots, n-2$ so that $\phi(v_{i-1}v_i)\equiv i$ (mod $n-2$). Let $\phi(vv_0)=\phi(vv_1)=n-1$ and $\phi(vv_2)=\phi(vv_{n-2})=\phi(vv_{n-1})=n$. Let $\mathcal I=\{\{v_0\},\{v_1\},\{v_2,\dots,v_{n-2}\},\{v_{n-1}\}\}$.
\item For $8\le n\le 11$, let $\phi(v_0v_1)=\phi(v_4v_5)=1$, $\phi(v_2v_3)=\phi(v_6v_7)=2$, $\phi(vv_{i-2})=\phi(v_{i-1}v_i)=\frac{i}{2}+2$ for $i=2,4,6,8$, $\phi(v_iv_{i-1})=i-2$ for $9\le i\le n$, and $\phi(vv_i)=n-1$ for $i=1,3,5,7$. Let $\mathcal I=\{\{v_1,v_2\},\{v_3,v_4\},\{v_5,v_6\},\{v_7,\dots, v_{n-1},v_0\}\}$.
\item For $n\ge 12$, let $\phi(v_0v_1)=\phi(v_6v_7)=1$, $\phi(v_3v_4)=\phi(v_9v_{10})=2$, $\phi(vv_{i-2})=\phi(v_{i-1}v_i)=\frac{2(i+1)}{3}+1$ for $i=2,5,8,11$, $\phi(vv_{i+1})=\phi(v_{i-1}v_i)=\frac{2i}{3}+2$ for $i=3,6,9$, $\phi(vv_1)=\phi(v_{11}v_{12})=10$, and $\phi(v_iv_{i-1})=i-2$ for $13\le i\le n$ (if $n\ge 13$). Let $\mathcal I=\{\{v_1,v_2,v_3\},$ $\{v_4,v_5,v_6\},\{v_7,v_8,v_9\},\{v_{10},\dots, v_{n-1},v_0\}\}$.
\end{itemize}
In each case, let $S$ be a set of three vertices of $W_n$. Then $S\cap I=\emptyset$ for some $I\in\mathcal I$. The cycle $C(\bar{I})$ is rainbow, and contains $S$. Hence, 
\[
crx_3(W_n)\le
\left\{
\begin{array}{l@{\quad}l}
n, & \textup{if $4\le n\le 7$,}\\[0.5ex]
n-1, & \textup{if $8\le n\le 11$,}\\[0.5ex]
n-2, & \textup{if $n\ge 12$.}
\end{array}
\right.
\]

Now, we prove the matching lower bounds for $crx_3(W_n)$. Consider a colouring of $W_n$, using fewer than $n$ colours for $4\le n\le 7$, and fewer than $n-1$ colours for $8\le n\le 11$. Then two edges of $C_n$, say $v_0v_1$ and $v_{i-1}v_i$ for some $2\le i\le \lfloor\frac{n}{2}\rfloor+1$, have the same colour. Any cycle containing the vertices $v_0,v_{\lfloor n/4\rfloor+1},v_{\lfloor n/2\rfloor+1}$ cannot be rainbow. Hence, $crx_3(W_n)\ge n$ for $4\le n\le 7$, and $crx_3(W_n)\ge n-1$ for $8\le n\le 11$. Now let $n\ge 12$. Again, we recall the result of Chartrand et al.~(\cite{COZ10}, Theorem 2.1) that $rx_3(C_n)=n-2$ for $n\ge 4$. If there exists a $3$-rainbow cycle colouring of $W_n$, using fewer than $n-2$ colours, then by omitting the edges $vv_i$, we obtain a colouring for $C_n$ which gives $rx_3(C_n)<n-2$, a contradiction. Hence, $crx_3(W_n)\ge n-2$.\\[1ex]
\indent (d) Since $W_n$ is Hamiltonian, we have $crx_k(W_n)\le n+1$. For $n\ge 2k$, define a colouring $\phi$ on $W_n$ as follows. For $1\le i\le n$, let $\phi(v_{i-1}v_i)=i$. For $0\le j\le 2k-1$ with $k$ even, or $0\le j\le 2k-7$ with $k$ odd, let 
\[
\phi(vv_j)=
\left\{
\begin{array}{l@{\quad}l}
j+2, & \textup{if $j\equiv 0,1$ (mod $4$),}\\[0.5ex]
j-1, & \textup{if $j\equiv 2,3$ (mod $4$).}
\end{array}
\right.
\]
In addition, for $k$ odd, let $\phi(vv_{2k-6})=\phi(vv_{2k-2})=2k-4$, $\phi(vv_{2k-5})=\phi(vv_{2k-1})=2k-2$, $\phi(vv_{2k-4})=2k-5$, and $\phi(vv_{2k-3})=2k-1$. We claim that $\phi$ is a $k$-rainbow cycle colouring. Let $S$ be a set of $k$ vertices of $W_n$. If $v\not\in S$, then $C_n$ is a rainbow cycle containing $S$. Now, let $v\in S$. For $k$ even, $S$ does not contain some vertex of the set $\{v_j: 0\le j\le 2k-1$ and $j\equiv 1,2$ (mod $4)\}$, say $v_\ell$. The cycle $C(\bar{I}_\ell)$ contains $S$, and is rainbow since we can obtain $C(\bar{I}_\ell)$ from $C_n$ by deleting and adding two edges with colours $\ell$ and $\ell+1$. For $k$ odd, $S$ does not contain some vertex of the set $\{v_j: 0\le j\le 2k-7$ and $j\equiv 1,2$ (mod $4)\}\cup\{v_{2k-5},v_{2k-2}\}$, or $S\cap\{v_{2k-4},v_{2k-3}\}=\emptyset$. If the former happens then, again letting $v_\ell$ be a vertex omitted by $S$, the cycle $C(\bar{I}_\ell)$ contains $S$ and is rainbow. If the latter, then the cycle $C(\overline{\{v_{2k-4},v_{2k-3}\}})$ contains $S$ and is rainbow, since to obtain $C(\overline{\{v_{2k-4},v_{2k-3}\}})$ from $C_n$, we delete three edges of colours $2k-4,2k-3,2k-2$, and add two edges of colours $2k-4$ and $2k-2$. Hence, $crx_k(W_n)\le n$.

Now, we prove the matching lower bounds for $crx_k(W_n)$. First, we show that $crx_k(W_n)\ge n$ for all $n$ with $n+1\ge k$. Suppose that there exists a $k$-rainbow cycle colouring, using fewer than $n$ colours. Then two edges $e_1$ and $e_2$ in $C_n$ have the same colour. Let $S$ be a set of $k$ vertices which contains the three (if $e_1$ and $e_2$ are incident) or four (otherwise) end-vertices of $e_1$ and $e_2$. Let $C$ be a rainbow cycle containing $S$. Then $C$ must use exactly one of $e_1$ and $e_2$, say $e_2$. Thus, $C$ must be the cycle obtained from $C_n$ by deleting $e_1$, and adding the two edges connecting $v$ and the end-vertices of $e_1$. But then, $C$ has length $n+1$, and so cannot be rainbow, a contradiction.

We are done for $n\ge 2k$. Now, let $n<2k$. We show that $crx_k(W_n)\ge n+1$. This clearly holds for $k=n+1$. Thus, let $k\le n$, so that $n\ge 4$. Suppose that there exists a $k$-rainbow cycle colouring of $W_n$, using  fewer than $n+1$ colours. By the previous argument, such a colouring uses exactly $n$ colours, and moreover, the $C_n$ must be rainbow coloured. We will choose a set $S$ of $k$ vertices of $W_n$ such that $v\in S$, so that any rainbow cycle containing $S$ has the form $C(I)$ for some $I\subset V'$. We say that a cycle $C(I)$ is \emph{large} if $|I|\ge \lceil\frac{n}{2}\rceil+1$, and \emph{small} otherwise. A rainbow coloured large cycle $C(I)$ is \emph{maximal} if there is no rainbow coloured cycle $C(I')$ such that $I'\varsupsetneq I$. By considering $S$ such that $v,v_0,v_{\lfloor n/2\rfloor},v_{\lfloor n/2\rfloor+1}\in S$, we see that there must exist a rainbow coloured large cycle, and hence a maximal cycle. Let $C(I_1),\dots,C(I_p)$ be all the maximal cycles, for some $p\ge 1$. Now, for two maximal cycles $C(I_i)$ and $C(I_j)$, the sets $\bar{I}_i$ and $\bar{I}_j$ are not nested, i.e., neither one contains the other. For $1\le i\le p$, let $v_{h_i}\in \bar{I}_i$ be the ``first vertex'' of $\bar{I}_i$, i.e., if $\bar{I}_i=\{v_q,v_{q+1},\dots\}$ for some $0\le q\le n-1$, then $h_i=q$. We may assume that $0=h_1<h_2<\cdots<h_p\le n-1$. Next, since $C_n$ is rainbow coloured and there are exactly $n$ colours, every edge $vv_i$ can belong to at most one of $C(I_1),\dots,C(I_p)$. Indeed, for the edge $vv_0$, the only possible maximal cycle containing $vv_0$, if it exists, is either $C(\{v_0,\dots, v_s\})$ for some $\lceil\frac{n}{2}\rceil\le s\le n-2$, or $C(\{v_t,\dots,v_{n-1},v_0\})$ for some $2\le t\le\lfloor\frac{n}{2}\rfloor$. A similar argument holds for any other edge $vv_i$. It follows that $p\le \lfloor\frac{n}{2}\rfloor\le k-1$. 
%
%We say that a cycle $C$ is \emph{maximal} if $C$ is rainbow coloured, $v\in V(C)$, and there is no rainbow coloured cycle $C'$ (containing $v$) such that $V(C')\varsupsetneq V(C)$. Also, a cycle $C$ with $v\in V(C)$ is \emph{large} if $|V(C)|\ge\lceil\frac{n}{2}\rceil+2$, otherwise $C$ is \emph{small}. Note that for any maximal cycle, since the length is at most $n$, a non-empty set $I$ of consecutive vertices of $C_n$ is omitted, and thus the maximal cycle is $C(I)$. Moreover, for any two maximal cycles, the omitted sets are not nested (neither one is contained in the other). Let $C(I_1),\dots,C(I_p)$ be all the large maximal cycles, for some $p\ge 0$. For $1\le j\le p$, let $v_{h_j}\in I_j$ be the ``first vertex'' of $I_j$, i.e., if $I_j=\{v_q,v_{q+1},\dots\}$ for some $0\le q\le n-1$, then $h_j=q$. Since the sets $I_1,\dots, I_p$ are pairwise non-nested, we may assume that $0=h_1<h_2<\cdots<h_p\le n-1$. Now, note that since the cycle $C_n$ is rainbow and we are using exactly $n$ colours, it is easy to see that every edge $vv_i$ can belong to at most one of the cycles $C(I_1),\dots,C(I_p)$. Indeed, for the edge $vv_0$, the only possible large maximal cycle containing $vv_0$, if it exists, is either $C(\{v_1,v_2,\dots, v_s\})$ for some $1\le s<\lfloor\frac{n}{2}\rfloor$, or $C(\{v_{n-1},v_{n-2},\dots, v_t\})$ for some $\lceil\frac{n}{2}\rceil< t\le n-1$. A similar argument holds for any other edge $vv_i$. It follows that $p\le \lfloor\frac{n}{2}\rfloor\le k-1$. 

We say that the distinct vertices $u_1,\dots,u_p$ are \emph{representing vertices} for $\bar{I}_1,\dots,\bar{I}_p$ if $u_i\in \bar{I}_i$ for $1\le i\le p$. Since $v_{h_1},\dots,v_{h_p}$ are representing vertices for $\bar{I}_1,\dots,\bar{I}_p$, there exists at least one such set of vertices. Let $S$ be a set of $k$ vertices with $v,u_1,\dots,u_p\in S$. If there is a large rainbow cycle containing $S$, then $S\subset V(C(I_i))$ for some $1\le i\le p$. But this implies $u_i\in V(C(I_i))$, a contradiction.
Hence, it suffices to show that we can choose a set $S$ of $k$ vertices, containing $v$ and representing vertices for $\bar{I}_1,\dots,\bar{I}_p$, so that no small rainbow cycle can contain $S$.\\[1ex]
\emph{Case 1.} $1\le p\le k-2$.\\[1ex]
\indent If $p=1$, then recall that $v_0\in \bar{I}_1$. Let $S$ be such that $v,v_0,v_{\lfloor n/2\rfloor},$ $v_{\lfloor n/2\rfloor+1}\in S$. If $2\le p\le k-2$, choose a set of representing vertices $u_1,\dots,u_p$ for $\bar{I}_1,\dots,\bar{I}_p$. We may assume that $u_1=v_0$ and $u_2=v_i$ where $1\le i\le\lfloor\frac{n}{2}\rfloor$. We further choose $v_{\lfloor n/2\rfloor+1}$ (if it is not already among the $u_j$), so that $v,u_1,u_2,\dots, u_p,v_{\lfloor n/2\rfloor+1}\in S$. In every case, there is no small cycle containing $S$.\\[1ex]
\emph{Case 2.} $p=k-1$.\\[1ex]
\indent We have $p=\lfloor\frac{n}{2}\rfloor$, and $n=2k-2$ or $n=2k-1$. We claim that, we can choose the representing vertices $u_1,\dots,u_p$ for $\bar{I}_1,\dots,\bar{I}_p$, so that they do not form a set of consecutive vertices of $V'$. For $1\le i\le p$, let $u_i=v_{h_i}$. If $u_1,\dots,u_p$ are consecutive vertices of $V'$, then we have $u_i=v_{i-1}$ for all $1\le i\le p=k-1$. Note that $u_p=v_{k-2}$, and we must have $v_{k-1}\in \bar{I}_p$. Otherwise, we have $\bar{I}_p=\{v_{k-2}\}$, and since $\bar{I}_1,\dots,\bar{I}_p$ are pairwise non-nested, it is easy to see that we have $\bar{I}_i=\{v_{i-1}\}$ for all $1\le i\le p$. But this is a contradiction, since $C(I_1)$ and $C(I_3)$ would both contain the edge $vv_1$. Thus, we may replace $u_p=v_{k-2}$ with $u_p=v_{k-1}$. We then have $u_1,\dots,u_p$ is a set of representing vertices for $\bar{I}_1,\dots,\bar{I}_p$ which is not a set of consecutive vertices of $V'$, since $k-1\le n-2$, so that $u_p=v_{k-1}$ is not adjacent to $v_0$ in $C_n$. 
%
%Again we choose representing vertices $u_1,\dots,u_p\in V(C_n)$ for $I_1,\dots,I_p$. Note that we can make the choice so that $u_1,\dots,u_p$ are not consecutive vertices of $C_n$. Otherwise, assume that they must be $v_0,\dots,v_{k-2}$. Then for $n=2k-2$, since $k\le n-2$, the edge $vv_k$ is not in a large maximal cycle, contradicting that all of $vv_0,\dots,vv_{n-1}$ are in the large maximal cycles. Similarly for $n=2k-1$, since $k+1\le n-2$, the edges $vv_k$ and $vv_{k+1}$ are not used in large maximal cycles, contradicting that there is exactly one edge among $vv_0,\dots,vv_{n-1}$ not in the large maximal cycles. 
%

Now let $S=\{v,u_1,\dots,u_p\}$. If $n=2k-2$, then $S$ is not contained in a small cycle, since this would require the $u_j$ to be consecutive vertices of $V'$. Let $n=2k-1$ and suppose that there is a small cycle containing $S$. We may assume that $u_1=v_0$, $u_p=v_{k-1}$, and $u_2,\dots,u_{p-1}\in\{v_1,\dots,v_{k-2}\}$, so that the only small cycle containing $S$ is $C=C(\{v_0,\dots,v_{k-1}\})$. The edges $vv_0$ and $vv_{k-1}$ are in $C$, and we may assume that $vv_0$ is contained in a maximal cycle $C^*$, so that $C^*=C(\{v_\ell,\dots,v_{n-1},v_0\})$ for some $2\le\ell\le k-1$ (since if $C^*$ contains $S$, then again we obtain a contradiction). It follows that the colour of $vv_0$ is the same as one of $v_0v_1,\dots, v_{\ell-1}v_\ell$, and thus $C$ is not a rainbow cycle.\\[1ex]
\indent We conclude that we can always find a set $S$ of $k$ vertices so that no rainbow cycle contains $S$, which is the required contradiction. Hence $crx_k(W_n)\ge n+1$, as required.
\end{proof}

Our next task is to determine $crx_k(G)$ when $G=K_n$ is a complete graph. We have the following result.

\begin{thm}\label{crxkKnthm}
\indent\\[-3ex]
\begin{enumerate}
\item[(a)] $crx_1(K_n)=crx_2(K_n)=3$ for $n\ge 3$.
\item[(b)] For $k\ge 3$, there exists $N=N(k)$ such that for $n\ge N$, we have $crx_k(K_n)=2k-1$.
\end{enumerate}
\end{thm}

\begin{proof}
%(a) It suffices to show that $crx_2(K_n)\le 3$ for every $n\ge 3$. We use induction on $n$ to show that there is a $2$-rainbow cycle colouring of $K_n$, using three colours. For $n=3$, we simply take the rainbow coloured $K_3$. Now let $n\ge 4$, and $u$ be a vertex of $G=K_n$. By induction, there exists a $2$-rainbow cycle colouring $\phi'$ for $G-u\cong K_{n-1}$, using three colours. Define a colouring $\phi$ for $G$ as follows. If $n-1$ is even, then take a perfect matching $M$ of $G-u$, and for $vw\in M$, let $\phi(uv)$ and $\phi(uw)$ be the two colours different from $\phi'(vw)$. If $n-1$ is odd, then take $x_1,x_2,x_3\in V(G-u)$ such that $\phi'(x_1x_2)$, $\phi'(x_2x_3)$ and $\phi'(x_3x_1)$ are distinct, and take a perfect matching $M'$ of $G-\{u,x_1,x_2,x_3\}$. For the colouring $\phi$, colour the edges from $u$ to the vertices of $M'$ in the same way as before, and let $\phi(ux_i)=\phi'(x_{i+1}x_{i+2})$ for $i=1,2,3$ (with indices taken modulo $3$). Clearly $\phi$ also uses three colours. Now by induction, any two vertices of $G-u$ lie in a rainbow triangle. If $v$ is a vertex of $M$ ($n-1$ even) or $M'$ ($n-1$ odd), then $u$ and $v$ lie in the rainbow triangle formed by $u$ and the edge of $M$ or $M'$ incident with $v$. For $n-1$ odd, $ux_ix_{i+1}$ is a rainbow triangle for $i=1,2,3$. Hence, $\phi$ is a $2$-rainbow cycle colouring of $G=K_n$, and we are done by induction.\\[1ex]
(a) It suffices to show that $crx_2(K_n)\le 3$ for every $n\ge 3$. We use induction on $n$ to show that there is a $2$-rainbow cycle colouring of $K_n$, using three colours. For $n=3$, we take the rainbow coloured $K_3$. Let $n\ge 4$, and $u$ be a vertex of $G=K_n$. By induction, there exists a $2$-rainbow cycle colouring $\phi'$ for $G-u\cong K_{n-1}$, using three colours. Define a colouring $\phi$ for $G$, using three colours, as follows. If $n-1$ is odd, take $x_1,x_2,x_3\in V(G-u)$ such that $\phi'(x_1x_2)$, $\phi'(x_2x_3)$ and $\phi'(x_3x_1)$ are distinct, and let $\phi(ux_i)=\phi'(x_{i+1}x_{i+2})$ for $i=1,2,3$ (with indices taken modulo $3$). Then (whether $n-1$ is odd or even), take a perfect matching $M$ on the remaining vertices of $G-u$, and for $vw\in M$, let $\phi(uv)$ and $\phi(uw)$ be the two colours different from $\phi'(vw)$. Now, any two vertices $x,y\in V(G)$ lie in a rainbow triangle under $\phi$, where we use induction if $x,y\in V(G-u)$. Hence, $\phi$ is a $2$-rainbow cycle colouring of $G$, and we are done by induction.\\[1ex]
\indent (b) Recall that for integers $r,\ell\ge 1$, the \emph{Ramsey number} $R_r(\ell)$ is the minimum integer $t$ such that, whenever we have an edge-colouring of $K_t$ with at most $r$ colours, there exists a monochromatic copy of $K_\ell$. It is well-known that all Ramsey numbers $R_r(\ell)$ exist. If $n\ge R_{2k-2}(k)$ and we have a $(2k-2)$-colouring of $K_n$, then there exists a monochromatic copy of $K_k$, say with vertex set $S$. Any cycle of length at most $2k-2$ containing $S$ must contain at least two edges with end-vertices in $S$, and so cannot be rainbow. Hence, $crx_k(K_n)\ge 2k-1$.

Now, we prove that $crx_k(K_n)\le 2k-1$ for $n$ sufficiently large. Consider a random $(2k-1)$-colouring of $K_n$. For $S\subset V(K_n)$ with $|S|=k$, let $E_S$ be the event that there is no rainbow cycle containing $S$. We estimate the probability $\mathbb P(E_S)$. Let $U_1,\dots, U_t\subset V(K_n)\setminus S$ be disjoint sets with $t=\lfloor\frac{n-k}{k-1}\rfloor$, and $|U_i|=k-1$ for $1\le i\le t$. Let $H$ be the complete subgraph of $K_n$ induced by $S$. Let $\Phi$ be the set of all $(2k-1)$-colourings of $H$ with no rainbow cycle $C_k$ containing $S$. Let $H$ be coloured by $\phi\in\Phi$ and $1\le i\le t$. There is at least one way to colour the $k(k-1)$ edges of $E(S,U_i)$ so that we have a rainbow cycle containing $S$, with vertices in $S\cup U_i$. Indeed, if $S=\{v_1,\dots , v_k\}$ and $U_i=\{u_1,\dots , u_{k-1}\}$, then we can assign $2k-1$ distinct colours to the cycle $C=v_1u_1\cdots v_{k-1}u_{k-1}v_kv_1$, since $v_1v_k$ is the only edge of $H$ in $C$. Hence, there are some $\alpha_\phi<(2k-1)^{k(k-1)}$ colourings of $E(S,U_i)$ so that there is no rainbow cycle containing $S$, with vertices in $S\cup B_i$. Let $\alpha=\max\{\alpha_\phi:\phi\in\Phi\}<(2k-1)^{k(k-1)}$. Since there are at most $k(k-2)+{n-k\choose 2}$ edges not in $E(H)\cup\bigcup_{i=1}^t E(S,U_i)$, we have
\[
\mathbb P(E_S)\le \frac{|\Phi|\alpha^t(2k-1)^{k(k-2)+{n-k\choose 2}}}{(2k-1)^{n\choose 2}}\le \frac{\gamma_k\alpha^{n/(k-1)}}{(2k-1)^{kn}}=\gamma_k^{\phantom{n}\!}\delta_k^n=o(n^{-k}),
\]
where $\gamma_k>0$ and $0<\delta_k<1$ are constants depending only on $k$.\\
\indent By the union bound,
\[
\mathbb P\bigg(\bigcup_{S\subset V(K_n),\,|S|=k}E_S\bigg)\le\sum_{S\subset V(K_n),\,|S|=k}\mathbb P(E_S)\le {n\choose k}o(n^{-k})<1,
\]
for $n$ sufficiently large. Hence, there exists a $k$-rainbow cycle colouring of $K_n$, using $2k-1$ colours, and $crx_k(K_n)\le 2k-1$.
\end{proof}

\noindent{\bf Remark.} We may similarly prove the result that $rx_k(K_n)=k$ for $n$ sufficiently large. This result is a special case of a conjecture of Chartrand et al.~(\cite{COZ10}, Conjecture 3.12), and the conjecture itself has since been solved by Cai et al.~\cite{CLS13}.\\[2ex]
%
%\begin{thm}[Special case of \cite{COZ10}, Conjecture 3.12 and \cite{CLS13}, Theorem 1]\label{rx3Knthm}
%For $k\ge 3$, there exists $N=N(k)$ such that for $n\ge N$, we have $rx_k(K_n)=k$.
%\end{thm}
%
%\begin{proof}
%If $n\ge R_{k-1}(k)$ and we have a $(k-1)$-colouring of $K_n$, then there exists a monochromatic copy of $K_k$, say with vertex set $S$. If $T$ is a tree with $V(T)\supset S$, then $T$ would be monochromatic if $V(T)=S$, and $e(T)\ge k$ if  $V(T)\varsupsetneq S$. Hence, we cannot have a rainbow tree containing $S$, and $rx_k(K_n)\ge k$.
%
%Next, we take a random $k$-colouring of $K_n$. For a subset $S$ of $K_n$ with $|S|=k$, let $E_S$ denote the event that there is no rainbow tree containing $S$. Let $H$ be the complete subgraph of $K_n$ induced by $S$. Then there are $\alpha_k<k^{k\choose 2}$ colourings of $H$ corresponding to $E_S$. For such a colouring of $H$ and $v\in V(K_n)\setminus S$, there are $k^k-k!$ colourings of the edges of $E(v,S)$ which correspond to $E_S$. It follows that
%\[
%\mathbb P(E_S)\le \frac{\alpha_k(k^k-k!)^{n-k}k^{n-k\choose 2}}{k^{n\choose 2}}=\frac{\beta_k(k^k-k!)^n}{k^{kn}}=o(n^{-k}),
%\]
%where $\beta_k$ is a constant depending only on $k$. Again by the union bound, we have $rx_k(K_n)\le k$ for $n$ sufficiently large.
%\end{proof}
%
\indent Next we consider $crx_k(G)$, when $G=K_{m,n}$ is a complete bipartite graph with $2\le m\le n$. A sufficient condition for the existence of $crx_k(K_{m,n})$ is $m\ge k$. We have the following result.

\begin{thm}\label{crxkKmnthm}
\indent
\begin{enumerate}
\item[(a)] $crx_1(K_{m,n})=4$ for $2\le m\le n$.
\item[(b)] 
\begin{enumerate}
\item[(i)] $crx_2(K_{2,n})=2n$ for $n\ge 2$.
\item[(ii)] $crx_2(K_{3,n})=r$ for $n\ge 36$, where $r$ is the integer such that ${r-1\choose 3}< n \le{r\choose 3}$. Thus, $crx_2(K_{3,n})=\sqrt[3]{6n}+C_n$, where $1<C_n<3$.
\item[(iii)] $crx_2(K_{m,n})=8$ for $m\ge 4$ and $n>7^m$.
\end{enumerate}
\item[(c)] Let $k\ge 2$.
\begin{enumerate}
\item[(i)] $\lceil\frac{n}{k-1}\rceil\le crx_k(K_{k,n})\le kn$ for $n\ge k$.
\item[(ii)] $crx_k(K_{m,n})\ge 4k$ for $m\ge 2k$ and $n>(k-1)(4k-1)^m$.
\item[(iii)] $crx_k(K_{m,n})\le 6k$ for $3k\le m\le n$.
\end{enumerate}
\end{enumerate}
\end{thm}

\begin{proof}
Throughout this proof, let $U$ and $V$ be the classes of $K_{m,n}$, with $U=\{u_1,\dots,u_m\}$ and $V=\{v_1,\dots,v_n\}$. For integers $1\le\ell\le N$, let $[N]=\{1,\dots,N\}$, and let $[N]^{(\ell)}$ denote the family of all subsets of $[N]$ with size $\ell$.\\[1ex]
\indent (a) Clearly $crx_1(K_{m,n})\ge g(K_{m,n})=4$. Now, define a $4$-colouring $\phi$ on $K_{m,n}$ where $\phi(u_1v_1)=1$; $\phi(u_1w)=2$ for $w\in V\setminus\{v_1\}$; $\phi(v_1w)=3$ for $w\in U\setminus\{u_1\}$; and $\phi(xy)=4$ for $x\in U\setminus\{u_1\}$, $y\in V\setminus\{v_1\}$. Then, every vertex of $K_{m,n}$ is in a rainbow cycle of length $4$ which uses $u_1v_1$. Hence, $crx_1(K_{m,n})\le 4$.\\[1ex]
\indent (b)(i) Since $K_{2,n}$ is minimally $2$-connected for $n\ge 2$, we have $crx_2(K_{2,n})=e(K_{2,n})=2n$ by Theorem \ref{crxk=m}(b). It is also not difficult to prove this result directly.\\[1ex]
%Clearly $crx_2(K_{2,n})\le 2n$, since the rainbow coloured $K_{2,n}$ is a $2$-rainbow cycle colouring. Now suppose that we have a colouring of $K_{m,n}$ with fewer than $2n$ colours. Then two edges $e$ and $f$ have the same colour. If $e$ and $f$ have distinct vertices in $V$, then let these be $x$ and $y$. Otherwise, if $e$ and $f$ have the same vertex in $V$, let $x$ be the this vertex, and $y\in V\setminus\{x\}$. In both cases, the only cycle containing $x$ and $y$ is $xuyu'x$, where $U=\{u,u'\}$. The cycle $xuyu'x$ contains $e$ and $f$, and hence is not rainbow. Therefore, $crx_2(K_{2,n})\ge 2n$.
\indent (b)(ii) Note that $n\ge 36$ implies $r\ge 8$. First, suppose that we have a colouring of $K_{3,n}$, using at most $r-1$ colours. Let the colours be $1,2,\dots, r-1$. For $x\in V$, let $A_x\in [r-1]^{(3)}$ be a set which contains the colours incident with $x$. Since $n>{r-1\choose 3}$, there exist $x,y\in V$ with $A_x=A_y$, and we cannot have $x$ and $y$ in a rainbow cycle. Thus, $crx_2(K_{3,n})\ge r$.

%Note that $n\ge 36$ implies $r\ge 8$. First, suppose that we have a $2$-rainbow cycle colouring of $K_{3,n}$, using at most $r-1$ colours, and let the colours be $1,2,\dots, r-1$. For $v\in V$, let $A_v\subset [r-1]$ be the set consisting of the colours incident with $v$. If $|A_v|=1$ for some $v\in V$, then there does not exist a rainbow cycle containing $v$. Also, observe that if $u,v\in V$ are such that $|A_u\cup A_v|\le 3$, then no rainbow cycle can contain $u$ and $v$. Thus, $|A_u\cup A_v|\ge 4$ for all $u,v\in V$. In particular, the sets $A_v$ are distinct and have size $2$ or $3$. Let $\mathcal A_j=\{A_v:|A_v|=j\}$ for $j=2,3$. Again by the observation, if we have a set $A_v\in\mathcal A_2$, then note that there are $r-3$ sets of $[r-1]^{(3)}$ containing $A_v$, and none of these sets may occur in $\mathcal A_3$. Moreover, for $A_u,A_v\in\mathcal A_2$, we have $A_u\cap A_v=\emptyset$, and no set of $[r-1]^{(3)}$ contains both $A_u$ and $A_v$. It follows that $|\mathcal A_3|\le {r-1 \choose 3}-|\mathcal A_2|(r-3)$, and therefore,
%\[
%n=|\mathcal A_2|+|\mathcal A_3|\le{r-1\choose 3}-|\mathcal A_2|(r-4)\le{r-1\choose 3}.
%\]
%This contradicts ${r-1\choose 3}<n$. Hence, $crx_2(K_{3,n})\ge r$.

Now, we define an $r$-colouring $\phi$ of $K_{3,n}$ as follows. Let $U=\{u_1,u_2,u_3\}$ and $V'=\{v_1,v_2,v_3\}\subset V$. Let $\phi(u_1v_3)=\phi(u_2v_1)=\phi(u_3v_2)=r-1$, $\phi(u_1v_2)=\phi(u_2v_3)=\phi(u_3v_1)=r$, and $\phi(u_iv_i)=r-5+i$ for $i=1,2,3$. For $1\le \ell\le r$, the \emph{colexicographic ordering} on the elements of $[r]^{(\ell)}$ is a linear ordering where an $\ell$-set $A=\{a_1,\dots,a_\ell\}$ with $a_1<\cdots<a_\ell$ precedes an $\ell$-set $B=\{b_1,\dots,b_\ell\}$ with $b_1<\cdots<b_\ell$ if $a_m<b_m$, where $m=\max\{p:a_p\neq b_p\}$. We label the $n-3$ vertices of $V\setminus V'$ with the $n-3$ initial members of $[r]^{(3)}$ in the colexicographic ordering. If a vertex $v\in V\setminus V'$ has been labelled with $\{a_1,a_2,a_3\}$ where $1\le a_1<a_2<a_3\le r$, then we let $\phi(u_iv)=a_i$ for $i=1,2,3$. For $v\in V$, let $A_v$ be the $3$-set consisting of the colours at $v$. Note that the vertices of $V'$ have been labelled with the last three members of $[r]^{(3)}$ in the colexicographic ordering, and since $n\le{r\choose 3}$, the $3$-sets $A_v$ are distinct over all $v\in V$.

We claim that $\phi$ is a $2$-rainbow cycle colouring. Firstly, it is easy to verify that any two vertices of $U\cup V'$ lie in a rainbow $C_4$. Next, let $x\in V$ and $y\in V\setminus V'$. If $|A_x\cup A_y|\in\{5,6\}$, then clearly $x$ and $y$ lie in a rainbow $C_4$. Thus, let $A_x=\{a,b,c\}$ and $A_y=\{a,b,d\}$, where $a,b,c,d\in[r]$ are distinct. If $x$ and $y$ are not in a rainbow $C_4$, then there exists a rainbow path of length four in the form $u_\alpha xu_\gamma yu_\beta$, where $\{\alpha, \beta, \gamma\}=\{1,2,3\}$, and using the colours $a,b,c,d$. This can be seen easily by considering the two edges of colours $c$ and $d$, and whether or not they are incident. It then suffices to find a vertex $w\in V\setminus\{x,y\}$ such that $\phi(wu_\alpha), \phi(wu_\beta)\not\in\{a,b,c,d\}$, so that $x$ and $y$ are in a rainbow $C_6$. If $a,b,c,d\le r-2$, then $x,y\in V\setminus V'$, and we can choose $w=v_\gamma$, so that $\{\phi(wu_\alpha), \phi(wu_\beta)\}=\{r-1,r\}$. Otherwise, suppose that at least one of $a,b,c,d$ is at least $r-1$. If $x\in V'$, then two of $a,b,c$ are $r-1$ and $r$, so that at least one of $a,b,d$, the colours at $y$, is $r-1$ or $r$. It follows that there exists $z\in V\setminus V'$ (either $z=x$ or $z=y$) such that $A_z$ contains $r-1$ or $r$. This implies that all members of $[r-2]^{(3)}$ are present among the $3$-sets labelling the vertices of $V\setminus V'$, since $[r-2]^{(3)}$ forms an initial segment in the colexicographic ordering, and a $3$-set containing $r-1$ or $r$ succeeds every $3$-set in $[r-2]^{(3)}$ in the ordering. Since $r-2\ge 6$, we can choose a suitable vertex $w\in V\setminus V'$ such that $A_w\in [r-2]^{(3)}$ and $A_w\subset [r-2]\setminus\{a,b,c,d\}$.

Finally, let $x\in U$ and $y\in V\setminus V'$. We choose $w\in V\setminus V'$ by choosing the $3$-set $A_w$ as follows. If $A_y=\{1,2,3\}$, then let $A_w=\{1,2,4\}$. If $A_y=\{a,b,c\}\neq\{1,2,3\}$ where $1\le a<b<c\le r$, then let $A_w=\{a-1,b,c\}$ if $a\ge 2$; or $A_w=\{a,b-1,c\}$ if $a=1$ and $b\ge 3$; or $A_w=\{a,b,c-1\}$ if $a=1$, $b=2$ and $c\ge 4$. Note that $A_w$ exists among the $3$-sets labelling $V\setminus V'$ since $A_w$ precedes $A_y=\{a,b,c\}$ in the colexicographic ordering. Then $y$ and $w$ are not in a rainbow $C_4$. Hence by the previous argument, $y$ and $w$ lie in a rainbow $C_6$, which thus contains $x$ and $y$.

Therefore, $\phi$ is a $2$-rainbow cycle colouring for $K_{3,n}$, and we have $crx_2(K_{3,n})\le r$.\\[1ex]
\indent (b)(iii) We have $crx_2(K_{m,n})\ge 8$ for $m\ge 4$ and $n>7^m$ from (c)(ii) below. We prove that $crx_2(K_{m,n})\le 8$ for $n\ge m\ge 4$. Define an $8$-colouring $\phi$ of $K_{m,n}$ as follows. Let $V'=\{v_1,v_2,v_3\}\subset V$, and assign to these vertices the $m$-vectors $\vec{v}_1=(4,3,2,1,1,\dots,1)$, $\vec{v}_2=(5,6,7,8,8,\dots,8)$ and $\vec{v}_3=(8,7,6,5,5,\dots,5)$. For every $v_j\in V\setminus V'$, assign the $m$-vector $\vec{v}_j=(1,2,3,4,4,\dots,4)$. For $v_j\in V$, let $\phi(u_iv_j)=\vec{v}_{ji}$ for $1\le i\le m$. We claim that $\phi$ is a $2$-rainbow cycle colouring for $K_{m,n}$. Let $U'=\{u_1,u_2,u_3\}\subset U$, and $x,y\in V(K_{m,n})$. We find a rainbow cycle containing $x$ and $y$. If $x\in U'$ and $y\in U$, we take $xv_1yv_2x$. If $x,y\in U\setminus U'$, we take $xwu_2v_1yv_2u_3v_3x$ for some $w\in V\setminus V'$. If $x\in V'$ and $y\in V$, we take $xu_1yu_2x$. If $x,y\in V\setminus V'$, we take $xu_2v_2u_3yu_4v_3u_1x$. Finally, if $x\in U$ and $y\in V$, then first choose $w\in U'\setminus\{x\}$, and note that the path $xyw$ is rainbow and uses colours either in $\{1,2,3,4\}$ or $\{5,6,7,8\}$. If the former, we can take $xywv_2x$ (note that $y\neq v_2$), and if the latter, we can take $xywv_1x$ (note that $y\neq v_1$). Hence, $crx_2(K_{m,n})\le 8$.\\[1ex]
\indent (c)(i) Let $n\ge k$. Clearly, we have $crx_k(K_{k,n})\le e(K_{k,n})=kn$. Now, suppose that we have a colouring of $K_{m,n}$, using fewer than $\lceil\frac{n}{k-1}\rceil$ colours. We may assume that $u_1v_1,\dots,u_1v_k$ have the same colour. Letting $S=\{v_1,\dots,v_k\}$, any cycle of $K_{k,n}$ containing $S$ must have the vertex set $U\cup S$, and so cannot be rainbow. Hence, $crx_k(K_{k,n})\ge \lceil\frac{n}{k-1}\rceil$.\\[1ex]
\indent (c)(ii) Let $m\ge 2k$ and $n>(k-1)(4k-1)^m$. Let $\phi$ be a $(4k-1)$-colouring of $K_{m,n}$, using the colours $1,2,\dots, 4k-1$. For $v_j\in V$, let $\vec{v}_j$ be the $m$-vector where $\vec{v}_{ji}=\phi(u_iv_j)$ for $1\le i\le m$. Then the lower bound on $n$ implies that, there exists a set of vertices $S\subset V$ with $|S|=k$, whose members are all assigned with the same $m$-vector. If there exists a rainbow cycle $C$ containing $S$, then in $C$, the distance between any two vertices of $S$ is at least $4$. Thus, $C$ has length at least $4k$, and so cannot be rainbow, a contradiction. Hence, $crx_k(K_{m,n})\ge 4k$.\\[1ex]
\indent (c)(iii) Let $3k\le m\le n$. We define a $6k$-colouring $\phi$ of $K_{m,n}$ as follows. Let $U'=\{u_1,\dots,u_{2k}\}\subset U$ and $V'=\{v_1,\dots,v_{2k}\}\subset V$. For $v_j\in V'$, we assign the $2k$-vector $\vec{v}_j=(2k-j+2,\dots,2k,1,2,\dots,2k-j+1)$, and let $\phi(u_iv_j)=\vec{v}_{ji}$ for $u_i\in U'$. For $u_i\in U\setminus U'$ and $v_j\in V'$, let $\phi(u_iv_j)=2k+j$.  For $u_i\in U'$ and $v_j\in V\setminus V'$, let $\phi(u_iv_j)=4k+j$. The edges $u_iv_j\in E(U\setminus U',V\setminus V')$ are coloured arbitrarily with used colours.

We claim that $\phi$ is a $k$-rainbow cycle colouring for $K_{m,n}$. Note that the vertices of $U\setminus U'$ are indistinguishable, and likewise for the vertices of $V\setminus V'$. Thus, it suffices to consider the subgraph $H$ of $K_{m,n}$ induced by $U'\cup U''\cup V'\cup V''$, where $U''=\{u_{2k+1},\dots, u_{3k}\}$ and $V''=\{v_{2k+1},\dots, v_{3k}\}$. Consider the edge set
\[
F=\{u_{2i-1}v_i:1\le i\le k\}\cup\{u_{2i}v_{k+i}:1\le i\le k-1\}\cup\{u_kv_{2k}\}.
\]
The edges of $F$ are rainbow coloured, since the colours of the edges in the three sets are $1,2,\dots,k$; $k+2,k+3,\dots,2k$; and $k+1$. Moreover, they almost form a perfect matching between $U'$ and $V'$: the edges $u_kv_{3k/2}$ (resp.~$u_kv_{\lceil k/2\rceil}$) and $u_kv_{2k}$ are incident at $u_k$ for $k$ even (resp.~odd), and the remaining edges are independent. Now, let $S\subset V(H)$ with $|S|=k$. We may assume that $S\cap V''\subset \{v_{2k+1},\dots,v_{3k-1}\}$, since the colouring $\phi$ restricted to $H$ has the same structure under the symmetry between the two classes of $H$; and $u_{2k}\not\in S$, since otherwise we may consider a shift of $F$. Relabelling the edges of $F$ as $\tilde{u}_1\tilde{v}_1,\dots,\tilde{u}_{2k-2}\tilde{v}_{2k-2},u_k\tilde{v}_{2k-1},u_kv_{2k}$, we obtain a rainbow cycle containing $S$ by connecting the edges $u_{2k+i}\tilde{v}_{2i}$, $u_{2k+i}\tilde{v}_{2i+1}$, $v_{2k+i}\tilde{u}_{2i-1}$ and $v_{2k+i}\tilde{u}_{2i}$ for $1\le i\le k-1$, and $u_{3k}\tilde{v}_1$ and $u_{3k}v_{2k}$. Hence, $crx_k(K_{m,n})\le 6k$.
%
%Let $A=S\cap U'$, $B=S\cap U''$, $C=S\cap V'$ and $D=S\cap V''$, with $|A|=a$, $|B|=b$, $|C|=c$ and $|D|=d$. Note that the colouring $\phi$ on $H$ has the same structure under the symmetry between the classes of $H$, and thus we may assume that $b\ge d$. Let $H'=H[U'\cup V']$. A \emph{one-to-two matching} of $H'$ is a subgraph $M$ such that every component is either a single edge, or a path $v_ju_iv_\ell$ for some $u_i\in U'$ and $v_j,v_\ell\in V'$, and $M$ is \emph{even} if the number of single edges is even. For $1\le i\le k$, let $X_i=\{u_{4i-3},u_{4i-2},u_{4i-1},u_{4i}\}$, $Y_i=\{v_{2i-1},v_{2i}\}$ and $Z_i=\{v_{2k+2i-1},v_{2k+2i}\}$. Let $H_i$ be the subgraph on $X_i\cup Y_i\cup Z_i$ with the edges $u_{4i-2}v_{2i-1},u_{4i-2}v_{2i},u_{4i-3}v_{2k+2i-1},u_{4i-1}v_{2k+2i}$, which have colours $2i,2i-1,2k+2i-1, 2k+2i$. We see that the union of any number of the graphs $H_i$ contains a one-to-two matching which is even and rainbow. Let $I=\{1\le i\le k:V(H_i)\cap(A\cup B)\neq\emptyset\}$, and $|I|=p$. Clearly, $p\le a+b$. We may assume that $|\{u_{4i}:i\in I, u_{4i}\in A\}|\le\frac{a}{4}$. We consider two cases.\\[1ex]
%\emph{Case 1.} $b\le\frac{k}{2}$.\\[1ex]
\end{proof}

We now consider $crx_k(G)$ when $G$ is a complete multipartite graph. For $t\ge 2$, let $K_{n_1,\dots,n_t}$ be the complete $t$-partite graph with class sizes $1\le n_1\le\cdots\le n_t$. Let $K_{t\times n}$ be the balanced complete $t$-partite graph, where each class has $n$ vertices. We have the following result.

\begin{thm}\label{crxkKtxnthm}
\indent
\begin{itemize}
\item[(a)] $crx_1(K_{n_1,\dots,n_t})=3$ for $t\ge 3$ and $1\le n_1\le\cdots\le n_t$.
\item [(b)] For $k,t\ge 2$, there exists $N=N(k,t)$ such that for $n\ge N$, we have $crx_k(K_{t\times n})=2k$.
\end{itemize}
\end{thm}

\begin{proof}
(a) It suffices to prove that $crx_1(K_{n_1,\dots,n_t})\le 3$. By Theorem \ref{crxkKnthm}(a), there exists a rainbow cycle colouring $\phi$ of $K_t$, using three colours. We may obtain a rainbow cycle colouring of $K_{n_1,\dots,n_t}$ with three colours by taking a blow-up of $\phi$, where the $i$th vertex of $K_t$ is replaced with a vertex set of size $n_i$, for $1\le i\le t$. \\[1ex]
\indent (b) Let $V_1,\dots,V_t$ be the classes of $K_{t\times n}$, with $|V_i|=n$ for $1\le i\le t$. If $n\ge k$, then we have $crx_k(K_{t\times n})\ge 2k$, since any cycle containing $k$ vertices of $V_1$ has length at least $2k$.

We prove that $crx_k(K_{t\times n})\le 2k$ for $n$ sufficiently large. Consider a random $2k$-colouring of $K_{t\times n}$. For $S\subset V(K_{t\times n})$ with $|S|=k$, let $E_S$ be the event that there is no rainbow cycle containing $S$. We estimate the probability $\mathbb P(E_S)$. For $1\le i\le t$, let $X_i=S\cap V_i$, where indices involving $i$ are taken modulo $t$ throughout. Let $s_i=|S\cap V_i|$, so that $s_1+\cdots+s_t=k$. Let $1\le q\le t$ be the number of the $s_i$ that are non-zero. We may assume that $s_1\ge\cdots\ge s_q\ge 1$. For $2\le i\le q+1$,  let $Y_{i1},\dots, Y_{ih}\subset V_i\setminus X_i$ be disjoint sets with $h=\lfloor\frac{n-s_2}{s_1}\rfloor$, and $|Y_{ij}|=s_{i-1}$ for all $1\le j\le h$. We may find such sets $Y_{ij}$ since $(\frac{n-s_2}{s_1})s_{i-1}\le n-s_i$ for every $2\le i\le \min(q+1,t)$, and $(\frac{n-s_2}{s_1})s_t\le n-s_1$ if $q=t$ and $n\ge k$, so that 
\[
|Y_{i1}|+\cdots+|Y_{ih}|=hs_{i-1}\le \Big(\frac{n-s_2}{s_1}\Big)s_{i-1}\le n-s_i=|V_i\setminus X_i|.
\] 
Define the subgraph $H=\bigcup_{i=1}^{q-1}[X_i,X_{i+1}]$. Let $\Phi$ be the set of all $2k$-colourings of $H$ with no rainbow cycle $C_k$ containing $S$. By convention, if $q=1$, we have $H$ is the empty graph on $S=X_1\subset V_1$, and $\Phi$ consists of one trivial colouring. Let $H$ be coloured by $\phi\in\Phi$. For $1\le j\le h$, define the subgraph $H_j=[X_1,X_2\cup Y_{2j}]\cup [X_q\cup Y_{qj},Y_{q+1,j}]\cup\bigcup_{i=2}^{q-1}[X_i\cup Y_{ij},X_{i+1}\cup Y_{i+1,j}]$, where $H_j=[X_1,Y_{2j}]$ if $q=1$. There is at least one way to colour the edges of $E(H_j)\setminus E(H)$ so that we have a rainbow cycle containing $S$, with vertices in $V(H_j)$. Indeed, if $q=1$, choose a spanning cycle of $H_j=[X_1,Y_{2j}]$ and assign $2k$ distinct colours to it. Let $2\le q\le t$. For $1\le i\le q$, choose a spanning path of $[X_i,Y_{i+1,j}]$ with end-vertices $x_i\in X_i$ and $y_{i+1}\in Y_{i+1,j}$. Then, add the edges $u_1u_2,v_2u_3,\dots,v_{q-1}u_q,v_qv_{q+1}$. We have a spanning cycle $C$ of $H_j$ where $u_1u_2$ is the only edge of $H$ in $C$, and thus we may assign $2k$ distinct colours to $C$. Hence, there are some $\alpha_\phi<(2k)^{e(H_j)-e(H)}=(2k)^{e(H_1)-e(H)}$ colourings of $E(H_j)\setminus E(H)$ so that there is no rainbow cycle containing $S$, with vertices in $V(H_j)$. Let $\alpha=\max\{\alpha_\phi:\phi\in\Phi\}<(2k)^{e(H_1)-e(H)}$. Since there are ${tn\choose 2}-e(H)-(e(H_1)-e(H))h$ edges not in $\bigcup_{j=1}^h E(H_j)$, we have
\[
\mathbb P(E_S)\le \frac{|\Phi|\alpha^h(2k)^{{tn\choose 2}-e(H)-(e(H_1)-e(H))h}}{(2k)^{tn\choose 2}}\le \frac{\alpha^{h}}{(2k)^{(e(H_1)-e(H))h}}=o(n^{-k}).
\]
\indent By the union bound,
\[
\mathbb P\bigg(\bigcup_{S\subset V(K_{t\times n}),\,|S|=k}E_S\bigg)\le\sum_{S\subset V(K_{t\times n}),\,|S|=k}\mathbb P(E_S)\le {tn\choose k}o(n^{-k})<1,
\]
for $n$ sufficiently large. Hence, there exists a $k$-rainbow cycle colouring of $K_{t\times n}$, using $2k$ colours, and $crx_k(K_{t\times n})\le 2k$.
\end{proof}

In Theorems \ref{crxkKmnthm} and \ref{crxkKtxnthm}, we were only able to determine or estimate $crx_k(G)$ for some values of $k$ and some complete bipartite or multipartite graphs $G$. In particular, for fixed $k$ and $m$, and $n$ sufficiently large, we see that $crx_k(K_{k,n})=\Theta_n(n)$, and $crx_k(K_{m,n})=O_n(1)$ for $m\ge 3k$. We pose the following problem.

\begin{prob}
Let $k\ge 2$ 
\begin{enumerate}
\item[(a)] For $2\le m\le n$, determine $crx_k(K_{m,n})$ for the cases not covered by Theorems \ref{crxkKmnthm} and \ref{crxkKtxnthm}. For fixed $k$ and $m$, what is the behaviour of $crx_k(K_{m,n})$ for $k\le m\le 3k$ and large $n$?
\item[(b)] For $t\ge 3$ and $1\le n_1\le\cdots\le n_t$, determine $crx_k(K_{n_1,\dots,n_t})$ for the cases not covered by Theorem \ref{crxkKtxnthm}.
\end{enumerate}
\end{prob}

We conclude this section by considering $crx_k(G)$ where $G=Q_n$ is the $n$-dimensional discrete binary cube. That is, $V(Q_n)$ consists of all $n$-vectors $v=(v_1,\dots,v_n)$ with $v_i\in\{0,1\}$ for $1\le i\le n$, and for $u,v\in V(Q_n)$, we have $uv\in E(Q_n)$ if and only if $u$ and $v$ differ at exactly one coordinate. Note that $Q_n$ is an $n$-regular graph with $2^n$ vertices and $2^{n-1}n$ edges. If $n=p+q$, then $Q_n$ can be considered as a blow-up of $Q_q$, with each vertex of $Q_q$ replaced by a copy of $Q_p$, and an edge of $Q_q$ is replaced by a matching that matches the corresponding vertices of $Q_p$. We write $Q_n=Q_p\oplus Q_q$, and for $v\in V(Q_n)$, we write $v=\hat{v}\oplus\tilde{v}$, where $\hat{v}\in V(Q_p)$ and $\tilde{v}\in V(Q_q)$ are the vectors containing the first $p$ and the last $q$ coordinates of $v$.

Since $Q_n$ is Hamiltonian for $n\ge 2$, the parameter $crx_k(Q_n)$ is defined for all $1\le k\le 2^n$ and $n\ge 2$. In the following result, we determine $crx_k(Q_n)$ exactly for small and large $k$.

\begin{thm}\label{crx123Qnthm}
Let $n\ge 2$.
\begin{enumerate}
\item[(a)] $crx_1(Q_n)=4$.
\item[(b)] $crx_2(Q_n)=crx_3(Q_n)=2n$.
\item[(c)] $crx_k(Q_n)=2^n$ for $2^{n-1}\le k\le 2^n$.
\end{enumerate}
\end{thm}

\begin{proof}
(a) Clearly $crx_1(Q_2)=4$. Let $n\ge 3$. We have $crx_1(Q_n)\ge g(Q_n)=4$. Since  $Q_n=Q_2\oplus Q_{n-2}$, we have a rainbow cycle colouring of $Q_n$ with four colours where each of the $2^{n-2}$ copies of $Q_2$ is given a rainbow colouring. Hence, $crx_1(Q_n)\le 4$.\\[1ex]
\indent (b) It suffices to prove that $crx_2(Q_n)\ge 2n$ and $crx_3(Q_n)\le 2n$. We have $crx_2(Q_n)\ge 2n$, since the shortest cycle containing the vertices $(0,\dots , 0)$ and $(1,\dots , 1)$ has length $2n$. To prove that $crx_3(Q_n)\le 2n$, we construct a $3$-rainbow cycle colouring of $Q_n$, using $2n$ colours. We proceed inductively. For $n=2$, we give $Q_2$ a rainbow colouring with colours $1,2,3,4$. Let $n\ge 3$, and suppose that we have constructed a $3$-rainbow cycle colouring $\phi'$ of $Q_{n-1}$, using the colours $1,\dots , 2(n-1)$. We have $Q_n=Q_{n-1}\oplus Q_1$, and for $v\in Q_n$, we have $v=\hat{v}\oplus \tilde{v}$, where $\hat{v}\in V(Q_{n-1})$ and $\tilde{v}\in V(Q_1)$. Let $Q_{n-1}^-$ and $Q_{n-1}^+$ be the copies of $Q_{n-1}$ containing the vertices $v\in V(Q_n)$ with $\tilde{v}=(0)$ and $\tilde{v}=(1)$, respectively. Define the colouring $\phi$ of $Q_n$, with colours $1,\dots , 2n$, as follows. For $uv\in E(Q_n)$, let
\[
\phi(uv)=\left\{
\begin{array}{ll}
\phi'(\hat{u}\hat{v}), & \textup{if }\tilde{u}=\tilde{v},\\[1ex]
2n-1, & \displaystyle\textup{if }\hat{u}=\hat{v},\,\tilde{u}\neq\tilde{v}\textup{, and }\sum \hat{u}_i=\sum \hat{v}_i\equiv 0\textup{ (mod }2),\\[1ex]
2n, & \displaystyle\textup{if }\hat{u}=\hat{v},\,\tilde{u}\neq\tilde{v}\textup{, and }\sum \hat{u}_i=\sum \hat{v}_i\equiv 1\textup{ (mod }2).
\end{array}
\right.
\]

We claim that $\phi$ is a $3$-rainbow cycle colouring for $Q_n$. Let $x,y,z\in V(Q_n)$. If $x,y,z\in V(Q_{n-1}^-)$ or $x,y,z\in V(Q_{n-1}^+)$, then by induction, $x,y$ and $z$ lie in a rainbow cycle. Now by the symmetry between $Q_{n-1}^-$ and $Q_{n-1}^+$, assume that $x,y\in V(Q_{n-1}^-)$ and $z\in V(Q_{n-1}^+)$. Let $x_+, y_+\in V(Q_{n-1}^+)$ and $z_-\in V(Q_{n-1}^-)$ where $x_+=\hat{x}\oplus (1)$, $y_+=\hat{y}\oplus(1)$ and $z_-=\hat{z}\oplus (0)$. By induction, there are rainbow cycles $C_-\subset Q_{n-1}^-$ and $C_+\subset Q_{n-1}^+$ such that, $C_-=a_1a_2\cdots a_ta_1$, $C_+=b_1b_2\cdots b_tb_1$ for some $t\ge 4$; $\hat{a}_p=\hat{b}_p$ and $\phi(a_pa_{p+1})=\phi(b_pb_{p+1})\in \{1,\dots , 2(n-1)\}$ for every $1\le p\le t$, where indices referring to $C_-,C_+$ are taken modulo $t$ throughout; and $x,y,z_-\in V(C_-)$ and $x_+,y_+,z\in V(C_+)$, with $(x,y,z_-)=(a_i,a_j,a_\ell)$ and $(x_+,y_+,z)=(b_i,b_j,b_\ell)$, where we may assume that $1\le i\le \ell< j\le t$. Let $P\subset C_-$ be the $y-x$ path $a_ja_{j+1}\cdots a_i$, and $R\subset C_+$ be the $x_+-y_+$ path $b_ib_{i+1}\cdots b_j$. Note that $z\in V(R)$. If $\{\phi(xx_+),\phi(yy_+)\}=\{2n-1,2n\}$, we have the rainbow cycle $xx_+Ry_+yPx$. If $\phi(xx_+)=\phi(yy_+)=2n-1$ or $2n$, then consider the vertices $a_{j-1}$ and $b_{j-1}$. Note that $a_{j-1}\neq x$ and $b_{j-1}\neq x_+$, since $C_-$ and $C_+$ have length at least $4$, and $\phi(xx_+)=\phi(yy_+)$ means that $x$ and $y$ are not neighbours in $C_-$, and $x_+$ and $y_+$ are not neighbours in $C_+$. Also, we have $\{\phi(xx_+),\phi(a_{j-1}b_{j-1})\}=\{2n-1,2n\}$. We have the rainbow cycle $xx_+Rb_{j-1}a_{j-1}yPx$.

Hence, $\phi$ is a $3$-rainbow cycle colouring for $Q_n$ with $2n$ colours, and we are done by induction.\\[1ex]
\indent (c) Let $2^{n-1}\le k\le 2^n$. Clearly we have $crx_k(Q_n)\le 2^n$ since $Q_n$ is Hamiltonian. Now, $Q_n$ is a balanced bipartite graph with classes $V_0$ and $V_1$, where $V_0$ (resp.~$V_1$) consists of the $2^{n-1}$ vertices at even (resp.~odd) distance from $(0,\dots,0)$. Let $S$ be a set of $k$ vertices containing $V_0$. Then, every cycle of $Q_n$ containing $S$ must have length $2^n$. Hence, $crx_k(Q_n)\ge 2^n$.
\end{proof}

For $4\le k<2^{n-1}$, the exact determination of $crx_k(Q_n)$ appears to be more difficult. In the following result, we give a partial answer to this problem, by showing that $crx_k(Q_n)$ is linear in $n$, for fixed $k\ge 4$ and $n$ sufficiently large. 

\begin{thm}\label{crx4Qnthm}
Let $k\ge 4$ and $n\ge 4k^2$. Then there exist constants $c_k,C_k>0$ (depending only on $k$) such that
\begin{equation}
c_kn\le crx_k(Q_n)\le C_kn.\label{crx4Qneq}
\end{equation}
Thus, we have $crx_k(Q_n)=\Theta_n(n)$.
\end{thm}

Before we prove Theorem \ref{crx4Qnthm}, we gather some tools in order to enable us to prove the upper bound. A graph $G$ is \emph{$k$-linked} if $|V(G)|\ge 2k$, and for every sequence $s_1, \dots , s_k, t_1, \dots , t_k$ of distinct vertices in $G$, there exist disjoint paths $P_1, \dots , P_k$ such that the end-vertices of $P_i$ are $s_i$ and $t_i$, for $1\le i\le k$. Clearly, a $k$-linked graph is also $k$-connected, but the converse is not true. Jung \cite{Jun70}, and Larman and Mani \cite{LM70}, showed that there exists a function $f(k)$ such that every $f(k)$-connected graph is $k$-linked, and it was shown in \cite{LM70} that $f(k)=2^{3k \choose 2}+2k$. Bollob\'as and Thomason \cite{BT96} then made a significant improvement, when they proved that the result holds for $f(k)=22k$. Thomas and Wollan \cite{TW05} further improved this to $f(k)=10k$.

\begin{thm}[\cite{TW05}, Corollary 1.3]\label{TWthm}
If $G$ is a $10k$-connected graph, then $G$ is k-linked.
\end{thm}

Let $S=(v_1,\dots,v_k)$ be an ordered $k$-tuple of vertices in a graph $G$, where repetitions are permitted for the vertices. That is, $S=(v_1,\dots,v_k)\in V(G)^k$. An \emph{$S$-subdivided closed walk} is a sequence $W=(v_1, P_1,v_2,P_2,v_3,\dots, v_k, P_k,v_1)$ such that, $P_i$ is a $v_i-v_{i+1}$ path for $1\le i\le k$, and no two $P_i$ and $P_j$ have an internal vertex in common. If $v_i=v_{i+1}$, then $P_i$ is trivial, and if $v_i\neq v_{i+1}$, then $P_i$ has length at least $2$. The indices of the vertices $v_i$ and the paths $P_i$ are taken modulo $k$ throughout. By using the above result, we have the following lemma, which says that for $Q_n$, if $n$ is sufficiently large, so that the vertex connectivity of $Q_n$ is large, then $Q_n$ contains an $S$-subdivided closed walk for every $S\in V(Q_n)^k$.

\begin{lem}\label{crx4Qnlem}
Let $k\ge 4$ and $n\ge 11k$. Then for all ordered $k$-tuple $S=(v_1,\dots,v_k)$ of vertices in $Q_n$, where repetitions of the vertices are permitted in $S$, there exists an $S$-subdivided closed walk in $Q_n$.
\end{lem}

\begin{proof}
Let $S=(v_1,\dots,v_k)\in V(Q_n)^k$ be an ordered $k$-tuple, and $u_1,\dots,u_\ell$ be the distinct members of $S$, for some $1\le \ell\le k$. Let $G=Q_n-\{u_1,\dots,u_\ell\}$. It is well-known that the vertex connectivity of $Q_n$ is $n$, and thus $G$ is $10k$-connected. By Theorem \ref{TWthm}, $G$ is $k$-linked.
%Also, 
%\[
%e(G)\ge 2^{n-1}n-\ell n\ge 2^{n-1}11k-11\ell k=2^n5k+2^{n-1}k-11\ell k>5k(2^n-\ell).
%\]

Now for $1\le p\le\ell$, let $m_p$ be the multiplicity of $u_p$ in $S$, so that $\sum_{p=1}^\ell m_p=k$. Since $u_p$ has at least $10k$ neighbours in $Q_n$ lying in $V(G)$, we can choose disjoint sets $U_1,\dots, U_\ell\subset V(G)$ such that $|U_p|=2m_p$, and all vertices of $U_p$ are neighbours of $u_p$ in $Q_n$, for $1\le p\le \ell$. We choose distinct vertices $s_1,t_1,\dots,s_k,t_k\in\bigcup_{p=1}^\ell U_p$ as follows. For $1\le i\le k$, let $v_i=u_p$ and $v_{i+1}=u_q$, and choose $s_i\in U_p$, $t_i\in U_q$. This can be done, since a vertex $u_p$ is considered exactly $2m_p=|U_p|$ times as we consider all consecutive pairs $(v_i,v_{i+1})$. Since $G$ is $k$-linked, we have disjoint paths $L_1,\dots,L_k$ in $G$ such that $L_i$ connects $s_i$ and $t_i$ for $1\le i\le k$. Define the paths $P_1,\dots, P_k$ as follows. For $1\le i\le k$, if $v_i=v_{i+1}$, then $P_i$ is trivial. Otherwise, if $v_i\neq v_{i+1}$, then let $P_i=v_is_iL_it_iv_{i+1}$, and note that $P_i$ has length at least $3$. Then, 
$(v_1,P_1,\dots,v_k,P_k,v_1)$ is an $S$-subdivided closed walk in $Q_n$, as required.
\end{proof}

We are now ready to prove Theorem \ref{crx4Qnthm}.

\begin{proof}[Proof of Theorem \ref{crx4Qnthm}]
We first prove the upper bound of (\ref{crx4Qneq}). Set $K=11k$. We show that the upper bound holds for $n\ge K=11k$ and $C_k=2^{2K-1}=2^{22k-1}$. We prove the stronger assertion that there exists a colouring of $Q_n$, using at most $C_kn$ colours, such that for every ordered $k$-tuple of vertices $S=(v_1,\dots,v_k)\in V(Q_n)^k$, there exists a rainbow $S$-subdivided closed walk $W$. Thus in particular, if $v_1,\dots,v_k$ are distinct, then $W$ is a rainbow cycle that visits $v_1,\dots,v_k$, in that order. 

Let $n=aK+b$, where $a\ge 1$ and $0\le b<K$. It suffices to construct the required colouring of $Q_n$, using at most $C_kaK$ colours. We use induction on $a$. For $a=1$, we have $n=K+b$, and $e(Q_n)\le e(Q_{2K-1})=2^{2K-2}(2K-1)<C_kK$. Thus, we may use a rainbow colouring for $Q_n$, and the strong assertion holds by Lemma \ref{crx4Qnlem}. Now, let $n=aK+b$ with $a\ge2$, and suppose that the strong assertion holds for $Q_{(a-1)K+b}$ and $Q_K$. We have $Q_{aK+b}=Q_{(a-1)K+b}\oplus Q_K$. By the induction hypothesis, there exist colourings $\eta$ on $Q_{(a-1)K+b}$ and $\theta$ on $Q_K$, using at most $C_k(a-1)K$ and $C_kK$ colours respectively, and each colouring satisfies the strong assertion. We define the colouring $\phi$ on $Q_{aK+b}$, using at most $C_kaK$ colours, as follows. Let $uv\in E(Q_{aK+b})$, with $u=\hat{u}\oplus\tilde{u}$ and $v=\hat{v}\oplus\tilde{v}$, where $\hat{u}, \hat{v}\in V(Q_{(a-1)K+b})$ and $\tilde{u},\tilde{v}\in V(Q_K)$. Let
\[
\phi(uv)=
\left\{
\begin{array}{ll}
\eta(\hat{u}\hat{v}), & \textup{if }\hat{u}\hat{v}\in E(Q_{(a-1)K+b})\textup{ and }\tilde{u}=\tilde{v},\\
\theta(\hat{u}\hat{v}), & \textup{if }\hat{u}=\hat{v}\textup{ and }\tilde{u}\tilde{v}\in E(Q_K).
\end{array}
\right.
\]
The colour sets used by $\eta$ and $\theta$ are disjoint. Thus, $\phi$ uses at most $C_kaK$ colours. We claim that $\phi$ is a colouring of $Q_{aK+b}$ that satisfies the strong assertion. Let $S=(v_1,\dots,v_k)\in V(Q_{aK+b})^k$, with $\hat{S}=(\hat{v}_1,\dots,\hat{v}_k)\in V(Q_{(a-1)K+b})^k$ and $\tilde{S}=(\tilde{v}_1,\dots,\tilde{v}_k)\in V(Q_K)^k$, where $\hat{v}_i\oplus\tilde{v}_i=v_i$ for $1\le i\le k$. By induction, there is a rainbow $\hat{S}$-subdivided closed walk $\hat{W}$ in $Q_{(a-1)K+b}$ under $\eta$, and a rainbow $\tilde{S}$-subdivided closed walk $\tilde{W}$ in $Q_K$ under $\theta$. Let $\hat{W}=(\hat{v}_1, L_1,\dots,\hat{v}_k,L_k,\hat{v}_1)$ and $\tilde{W}=(\tilde{v}_1, M_1,\dots,\tilde{v}_k,M_k,\tilde{v}_1)$. Now for consecutive vertices $v_i$ and $v_{i+1}$, where $1\le i\le k$, we define the $v_i-v_{i+1}$ path $P_i$ in $Q_{aK+b}$, as follows. 
\begin{itemize}
\item If $v_i=v_{i+1}$, then $P_i$ is trivial.
\item If $v_i\neq v_{i+1}$, with $\hat{v}_i=\hat{v}_{i+1}$, then let $P_i$ be the copy of $M_i$ that connects $v_i$ and $v_{i+1}$. Note that $P_i$ has length at least $2$.
\item Let $v_i\neq v_{i+1}$, with $\hat{v}_i\neq\hat{v}_{i+1}$. Then $L_i$ has length at least $2$. There exist copies of $L_i$ within the copies of $Q_{(a-1)K+b}$ that contain $v_i$ and $v_{i+1}$. Let $L'$ and $L''$ be these two copies of $L_i$. Let $u$ be an internal vertex of $L_i$, and $u',u''$ be the copies of $u$ in $L'$ and $L''$. Let $M$ be the copy of $M_i$ that connects $u'$ and $u''$. Note that if $v_i$ and $v_{i+1}$ lie in the same copy of $Q_{(a-1)K+b}$, i.e., $\tilde{v}_i=\tilde{v}_{i+1}$, then $L'=L''$, $u'=u''$, and $M$ is trivial. Let $P_i=v_iL'u'Mu''L''v_{i+1}$, which has length at least $2$.
\end{itemize}
Now, let $W=(v_1,P_1, \dots,v_k,P_k, v_1)$. Then $W$ uses exactly the colours of $\hat{W}$ and $\tilde{W}$, so $W$ is rainbow coloured in $Q_{aK+b}$ under $\phi$. We claim that the paths $P_i$ are pairwise internally vertex-disjoint. It suffices to consider two paths of the type as described in the third item above. Let $P_i=v_iL'u'Mu''L''v_{i+1}$ and $P_j=v_jL^\ast u^\ast N u^{\ast\ast} L^{\ast\ast}v_{j+1}$ be two such paths, for some $1\le i<j\le k$ with $v_i\neq v_j$. By considering the first $(a-1)K+b$ coordinates of the vertices of $V(v_iL'u'\cup u''L''v_{i+1})\setminus\{v_i,v_{i+1}\}$ and $V(v_jL^\ast u^\ast\cup u^{\ast\ast}L^{\ast\ast}v_{j+1})\setminus\{v_j,v_{j+1}\}$, we obtain the internal vertices of $L_i$ and $L_j$, so these two sets of vertices are disjoint. By considering the last $K$ coordinates of the internal vertices of $M$ and $N$, we see that these two sets are disjoint, since we obtain the internal vertices of $M_i$ and $M_j$. Similarly,  $V(v_iL'u'\cup u''L''v_{i+1})\setminus\{v_i,v_{i+1}\}$ is disjoint from the set of internal vertices of $N$, since the last $K$ coordinates of any vertex in the former set is either $\tilde{v}_i$ or $\tilde{v}_{i+1}$, while those of any vertex in  the latter set are neither. The claim follows, and hence $W$ is a rainbow $S$-subdivided closed walk in $Q_{aK+b}$. This completes the induction step for the strong assertion, and the proof of the upper bound of (\ref{crx4Qneq}).

Now, for the lower bound of (\ref{crx4Qneq}), we prove that $crx_k(Q_n)>\frac{kn}{2}$ for $n\ge 4k^2$. We find $k$ vertices in $Q_n$ such that the distance between every pair is greater than $\frac{n}{2}$. Then the shortest cycle containing the $k$ vertices has length greater than $\frac{kn}{2}$, and the assertion follows. Let $k'=2^{\lceil\log_2k\rceil}\ge 4$, the smallest power of $2$ that is at least $k$. A \emph{Hadamard matrix} is a $q\times q$ $(-1,1)$-matrix such that any two column vectors are orthogonal. Equivalently, for any two column vectors $x$ and $y$, the number of coordinates where $x$ and $y$ agree equals the number of coordinates where $x$ and $y$ differ. The famous \emph{Hadamard conjecture} states that whenever $q$ is a multiple of $4$, then there exists a $q\times q$ Hadamard matrix. We employ  \emph{Sylvester's construction}, which gives the existence of a Hadamard matrix for every $q$ that is a power of $2$. The construction is given by the following recursion.
\[
H_1=(1),\textup{ and }H_{2^t}=
\begin{pmatrix}
H_{2^{t-1}} & H_{2^{t-1}}\\
H_{2^{t-1}} & -H_{2^{t-1}}
\end{pmatrix}
\textup{ for }t\ge 1. 
\]

Consider the $k'\times k'$ Hadamard matrix $H_{k'}$. Every entry in the first row is a $1$. Delete the first row of $H_{k'}$, replace every $-1$ with a $0$, and take $k$ of the $k'$ resulting column $(k'-1)$-vectors, say $c_1,\dots,c_k$. Let $A_1\dcup\cdots\dcup A_{k'-1}$ be a partition of $\{1,\dots,n\}$, where $|A_j|=\lfloor\frac{n}{k'-1}\rfloor$ or $|A_j|=\lceil\frac{n}{k'-1}\rceil$ for every $1\le j\le k'-1$. Let $v_1,\dots,v_k$ be the $n$-vectors where $v_{\ell i}=c_{\ell j}$ for $1\le \ell\le k$, $1\le i\le n$, and $i\in A_j$. Any two of $c_1,\dots,c_k$ have exactly $\frac{k'}{2}$ differing coordinates, so the distance between any two of $v_1,\dots,v_k$ is at least $\frac{k'}{2}\lfloor\frac{n}{k'-1}\rfloor>\frac{n}{2}$, since $n\ge 4k^2>(k'-1)^2$. Thus, any cycle of $Q_n$ containing $v_1,\dots, v_k$ has length greater than $\frac{kn}{2}$. 

Therefore, both inequalities of (\ref{crx4Qneq}) hold for $n\ge 4k^2$. The proof of Theorem \ref{crx4Qnthm} is complete.
\end{proof}

\noindent{\bf Remark.} Many other problems concerning the existence of rainbow cycles in edge-coloured discrete binary cubes have also been considered. A result of Faudree et al.~\cite{FGLS93} states that, for $n\ge 4$, $n\neq 5$, there exists an edge-colouring of $Q_n$ with $n$ colours such that, every copy of $C_4$ is rainbow. Since such an edge-colouring of $Q_n$ must be proper, $n$ colours are also necessary. This result was extended by Mubayi and Stading \cite{MS13}. They proved that for $\ell\equiv 0$ (mod $4$), there exists an edge-colouring of $Q_n$ with $\Theta_n(n^{\ell/4})$ colours such that, every copy of $C_\ell$ is rainbow, and moreover, $\Theta_n(n^{\ell/4})$ colours are also necessary. These may be considered as edge-chromatic number type results, and they  appear to be closely related to our problem of the determination of $crx_k(Q_n)$.\\[2ex]
\indent In Theorem \ref{crx4Qnthm}, we have only determined the correct order of magnitude of $crx_k(Q_n)$ for fixed $k$ and large $n$, and we have made no serious attempt to optimise the lower bound on $n$ (in terms of $k$), or the terms $c_k$ and $C_k$. We have $c_k=\frac{k}{2}$ and $C_k=2^{22k-1}$, so the gap between $c_k$ and $C_k$, when considered as a function of $k$, is very large. It would be desirable to reduce this large gap. We leave the determination of $crx_k(Q_n)$ as an open problem.

\begin{prob}
For $n\ge 4$, determine $crx_k(Q_n)$ for every $4\le k< 2^{n-1}$.
\end{prob}

\section*{Acknowledgement}

%Y.~Mao is supported by the National Science Foundation of China (No.~11161037) and the Science Foundation of Qinghai Province (No.~2014-ZJ-907). 

Henry Liu is partially supported by the Startup Fund of One Hundred Talent Program of SYSU, National Natural Science Foundation of China (No.~11931002), and National Key Research and Development Program of China (No. 2020YFA0712500).

\end{document}